\documentclass[11pt, letterpaper, reqno]{amsart}
\usepackage{amssymb, amsmath, amsthm, color, marginnote, hyperref, cancel ,amscd, amsfonts,mathtools}
\usepackage{enumitem}
\usepackage[all]{xy}
\usepackage{hyperref}
\usepackage[T1]{fontenc} 
\usepackage{mathrsfs}
\usepackage{ wasysym}

\usepackage{multicol}
\usepackage{rotating}
\usepackage{graphicx,fancyhdr}
\usepackage{multirow}

\hypersetup{colorlinks=true,citecolor=cyan,linkcolor=cyan,linktocpage=true}

\newtheorem{theorem}{Theorem}
\newtheorem*{characteristic}{Theorem \ref{characteristic}}
\newtheorem*{structure1}{Theorem \ref{structure1}}
\newtheorem*{simplequot}{Theorem \ref{simplequot}}
\newtheorem*{mainA}{Theorem A}
\newtheorem*{mainB}{Theorem B}
\newtheorem*{corSW}{Corollary \ref{SW}}
\newtheorem*{cor3cases}{Corollary \ref{3cases}}

\newtheorem{SWconjecture}[theorem]{Sierra-Walton Conjecture}
\newtheorem{cor}[theorem]{Corollary}
\newtheorem*{infcor}{Informal Corollary}
\newtheorem{lemma}[theorem]{Lemma}
\newtheorem{prop}[theorem]{Proposition}

\theoremstyle{definition}
\newtheorem{remark}[theorem]{Remark}

\newcommand{\pf}{\begin{proof}}
\newcommand{\epf}{\end{proof}}

\newcommand{\ad}{\operatorname{ad}}

\newcommand{\Der}{\operatorname{Der}}
\newcommand{\End}{\operatorname{End}}
\newcommand{\rad}{\operatorname{rad}}

\newcommand{\Hom}{\operatorname{Hom}}

\newcommand{\rk}{\operatorname{rk}}

\newcommand{\Q}{{\mathbb Q}}

\newcommand{\Z}{{\mathbb Z}}

\newcommand{\cC}{\mathcal C}
\newcommand{\cD}{\mathscr D}

\newcommand{\cF}{\mathcal F}

\newcommand{\cL}{\mathcal L}
 
\newcommand{\cP}{\mathcal P}
\newcommand{\cQ}{\mathcal Q}

\newcommand{\Aut}{\mbox{\rm Aut\,}}

\newcommand{\g}{\mathfrak g}
\newcommand{\G}{\mathfrak G}
\newcommand{\fgl}{\mathfrak gl}

\newcommand{\fso}{\mathfrak so}
\newcommand{\fsp}{\mathfrak sp}

\newcommand{\Cs}{\mathscr C}

\newcommand{\fsl}{\mathfrak{sl}}

\newcommand{\fa}{\mathfrak{a}}
\newcommand{\fb}{\mathfrak{b}}

\newcommand{\fg}{\mathfrak{g}}
\newcommand{\fh}{\mathfrak{h}}
\newcommand{\fH}{\mathfrak{H}}

\newcommand{\fm}{\mathfrak{m}}

\newcommand{\fp}{\mathfrak{p}}

\newcommand{\fq}{\mathfrak{q}}

\newcommand{\fr}{\mathfrak{r}}
\newcommand{\fs}{\mathfrak{s}}

\newcommand{\fz}{\mathfrak{z}}

\newcommand{\fB}{\mathfrak{B}}
\newcommand{\fG}{\mathfrak{G}}

\newcommand{\fM}{\mathfrak{M}}
\newcommand{\fS}{\mathfrak{S}}

\newcommand{\Ad}{\mathrm{Ad}}

\newcommand{\Ann}{\mathrm{Ann}}
\newcommand{\Ass}{\mathrm{Ass}}

\newcommand{\Card}{\mathrm{Card\,}}
\newcommand{\Cent}{\mathrm{Cent}}
\newcommand{\Cl}{\mathrm{Cl}}
\newcommand{\codim}{\mathrm{codim}}

\newcommand{\Com}{\mathrm{Com}}

\newcommand{\Deg}{\mathrm{Deg\,}}
\newcommand{\diag}{\mathrm{diag}}

\newcommand{\Elem}{\mathrm{Elem}}

\newcommand{\F}{\mathrm{F}}

\newcommand{\Free}{\operatorname{Free}}
\newcommand{\Gal}{{\mathrm {Gal}}}

\newcommand{\Ker}{\operatorname{Ker}}

\newcommand{\Mat}{\mathrm{Mat}}

\newcommand{\Out}{\mathrm{Out}}

\newcommand{\Spec}{{\mathrm {Spec}}\,}

\newcommand{\Tr}{{\mathrm {Tr}}}

\newcommand{\Vect}{\mathrm{Vect}}
\newcommand{\Vir}{\mathrm{Vir}}

\newcommand{\Witt}{\mathrm{Witt}}

\renewcommand{\d}{\mathrm{d}}

\begin{document}
\title[Weakly Noetherian  Lie algebras]{Weakly Noetherian   Lie algebras\\and the Sierra-Walton Conjecture}

\author[O. Mathieu]{Olivier Mathieu}
\address[O.~Mathieu]{CNRS,
Institut Ca\-mille Jordan
UMR 5028 du CNRS, 
Universit\'e Claude Bernard Lyon
69622 Villeurbanne Cedex, France}
\email{mathieu@math.univ-lyon1.fr}

\subjclass{17B35, 17B68, 16P99, 16R40}
\keywords{Infinite dimensional Lie algebras, Sierra-Walton conjecture, Krichever-Novikov Lie algebras,  Affine Lie Aalgebras, Central identities}

\address{
SUSTech, Shenzhen International Center for Mathematics, Shenzhen, China}

\begin{abstract} Let $K$ be a field of characteristic zero. Motivated by the conjecture that an enveloping algebra
$U(\fg)$ is Noetherian only if $\fg$ is finite dimensional, we define the notion of weakly Noetherian Lie algebras.

The main result, Theorem A, states that weakly Noetherian Lie algebras have a very constrained structure. In the specific case of $\Z^n$ graded Lie algebras, it implies an explicit classification of the perfect strictly weakly Noetherian Lie algebras, stated in Theorem B. The proofs of both theorems are quite long, and uses concrete results due to  Tits,
Formanek, Razmyslov, Grabowski and the  author.

The first theorem provides some insight on the desired conjecture. The second one implies the  conjecture for
all perfect $\Z^n$-graded Lie algebras, improving a celebrated theorem of Sierra and Walton.

\end{abstract}

\thanks{\noindent  \emph{2020 MSC.} 16S30, 17B35.
N.~A. was  partially supported by the Secyt (UNC),
Research  partially supported
by UMR 5028 du CNRS and
by the Shenzhen International Center of Mathematics, SUSTech.}

\maketitle
\setcounter{tocdepth}{2}
\tableofcontents

\section*{Introduction} 

\subsection{General introduction}
Over  a field $K$ of characteristic 0,
any  finite dimensional Lie algebra $\g$
admits a Levi decomposition 
$\fg\simeq\fs\ltimes \fr$, where $\fs$ is 
semisimple and $\fr$ is a solvable ideal.
Motivated by the conjecture 
mentionned in the abstract, we pose the 
 very naive question:

\hskip1cm{\it Do infinite dimensional Lie algebras
admit a structure }

\hskip1cm {\it  theory similar to the finite dimensional case?}

\noindent In general the answer is obviously negative, but we can expect a positive answer for some classes of infinite dimensional Lie algebras. For example, Aldosray, Amayo and Stewart
\cite{ AS1,AS2,Amayo-Stewart, Stewart}
has proved some general results for Lie algebras satisfying some Artinian-Noetherian  conditions. 
Here we try a different approach:
we say that a Lie algera $\fG$ is {\it weakly   Noetherian} if
\begin{enumerate}
\item[(a)] for any subalgebra $\fm\subset\fG$,
the space $\fm/
[\fm,\fm]$ is finite dimensional.
\end{enumerate}
If, in addition
\begin{enumerate}
\item[(b)] any chain of 
characteristic ideals in $\fG$ stabilizes,
\end{enumerate}
\noindent $\fG$ is called {\it Noetherian}. Moreover $\fG$ is called
{\it strictly weakly Noetherian} if $\fG\otimes L$ is weakly Noetherian, for any finite extension $L$ of $K$.
Typical examples of Noetherian Lie algebras are the Krichever-Novikov Lie algebras $\Vect_X$ of vector fields on a smooth algebraic curve $X$, as it is proved in
Section \ref{ExKN}. 

At first glance, the notion of 
weakly Noetherian Lie algebras looks like an abstract nonsense notion.
In fact, we will see that it
is connected with interesting concrete algebraic theories, as explained  in Subsection
\ref{idea}.

First we consider the case of $\Z^n$-graded Lie algebras,
$\cL=\oplus_{{\bf m}\in\Z^n}\,\cL_{\bf m}$, where is is assumed that all homogenous components have finite dimensions.  Given a finite extension $L$ of $K$, we define

\begin{align*}
\Witt(L)&=\Der\,L[t]\\
\Vir(L)&=\Der\,L[t,t^{-1}],
\end{align*}

\noindent and we denote as  $\widehat{\Vir}(L)$  the Virasoro algebra, that is the universal central extension of $\Vir(L)$. For 
perfect $\Z^n$-graded Lie algebra, there is an explicit classification, up to a finite dimensional factor.

\begin{mainB}
Let $\cL$ be a perfect strictly 
weakly Noetherian
$\Z^n$-graded Lie algebra
and let $\hat{\cL}$ be its its universal central extension. Then we have

\begin{align*}
\widehat{\cL}\simeq \fg\,\oplus 
\,\big[\oplus_{i=1}^n\Witt(E_i)\big]\,\oplus 
\,\big[\oplus_{j=1}^m\widehat{\Vir}(F_j)\big]
\end{align*}

\noindent where $\fg\simeq \fs\ltimes\fr$ is a  (perfect) finite dimensional Lie algebra,
$n$, $m$ are integers and 
$E_1,\ldots,E_n,F_1,\ldots, F_m$ are finite extensions of $K$.
\end{mainB}

In order to state the  structure result in 
general we first need some additional definitions.
Recall that the {\it radical} $\rad(\fG)$ of a Lie algebra
$\fG$ is the intersection of all maximal ideals $\fm$ of codimension $>1$. Assume now that
$\fG$ be a weakly Noetherian Lie algebra satisfying 
the following simplifying hypothesis

\begin{enumerate}
\item[(h1)] $\fG=[\fG,\fG]$, and 
\item[(h2)] $\fG$ does not admit finite dimensional simple quotients.
\end{enumerate}

\noindent Then we define, by transfinite induction, a decreasing sequence,
indexed by ordinals $\alpha$,  
of Lie subalgebras 
$$\fG=\fG_{(0)}\supset \fG_{(1)}\supset \ldots\fG_{(\alpha)}\supset\ldots$$ 

\noindent as follows:

\begin{enumerate}
\item[(a)] if the ordinal $\alpha$ is a successor,
that is $\alpha=\beta+1$ for some $\beta$, then
$$\fG_{(\alpha)}=
[\rad(\fG_{(\beta)}),\rad(\fG_{(\beta)})]$$
\item[(b)] If $\alpha>0$ is a limit ordinal, then
$$\fG_{(\alpha)}=[\cap_{\beta<\alpha}\,\rad(\fG_{(\beta)}), \cap_{\beta<\alpha}\,\rad(\fG_{(\beta)})].$$
\end{enumerate}

In its simplified form, 
the main result is as follows:

\begin{mainA} (simplified version) Let $\fG$ be a weakly Noetherian Lie algebra 
satisfying (h1) and (h2).
Then, for any ordinal $\alpha$:

\begin{enumerate}
\item[(a)] the Lie algebra $\fG_{(\alpha)}$
is a  perfect ideal,

\item[(b)] 
$\fG_{(\alpha)}/\fG_{(\alpha+1)}$
is a central extension of 
$\fG_{(\alpha)}/\rad\fG_{(\alpha)}$ by a finite
dimensional center,

\item[(c)] any
simple quotient  of $\fG_{(\alpha)}$ is infinite dimensional,
and

\item[(d)]
for $\alpha\geq 1$, no Krichever-Novikov algebras occur as a quotient of $\fG_{(\alpha)}$.
\end{enumerate}
\smallskip
Moreover if $\fG$ is Noetherian,
then $\fG_{(\alpha)}=0$ for $\alpha$ big enough.
\end{mainA}

Therefore, Theorem A shows that Noetherian Lie algebras $\fG$ have a very
constrained structure. Roughly speaking, $\fG$
is a "tower" of central extensions of 
 infinite dimensional  simple Lie algebras.

\subsection{Relation with the Sierra-Walton conjecture}

\noindent Our investigation is motivated by the following:

\begin{SWconjecture}\label{conjecture:enveloping-Noetherian} \cite{Sierra-Walton}
If $\fg$ is infinite dimensional, then $U(\fg)$ is  not Noetherian.
\end{SWconjecture}

This question has been stated by several authors since the 70's, among them 
R. K. Amayo, I. Stewart, J. Dixmier, V. Latishev, K. A. Brown \cite{Amayo-Stewart, Brown} 
A  breakthrough result was obtained in 2013 by S. Sierra and Ch. Walton:

\begin{theorem}\label{SWthm} \cite{Sierra-Walton}
The enveloping algebra of the Witt algebra $W(K)$
is not Noetherian.
\end{theorem}

Consequently, S. Sierra and Ch. Walton stated this question as a conjecture.
From Sierra-Walton Theorem, one can derived the same result for all Kri\-che\-ver-Novikov algebras \cite{Buzaglo}.
The conjecture was confirmed recently for some other classes of Lie algebras 
\cite{BellBuz1,BellBuz2} and \cite{AndMath}. The old work \cite{U} is also 
connected with these questions. 
As a consequence of Theorem B, we obtain the following generalization of Sierra-Walton Theorem.

\begin{corSW} Let $\cL$ be a perfect 
$\Z^n$-graded Lie algebra. 
The algebra $U(\cL)$ is left Noetherian
only if $\cL$ has finite dimension.
\end{corSW}

A refined version of Theorem A is proved in section
\ref{proofA}. As a consequence we deduce:

\begin{cor3cases}  Any infinite dimensional Noetherian Lie algebra $\fG$ contains two characteristic ideals $\fp\supset\fq$ with
\begin{align*}
\codim&\, \fp<\infty & \codim\,\fq=\infty
\end{align*} 
\noindent such that the  Lie algebra $\fp/\fq$ 
is just-infinite and  either

\begin{enumerate}
\item[(type A)] $\fp/\fq$ is simple, or

\item[(type B)] $\fp/\fq$ is residually nilpotent, or

\item[(type C)] $\fp/\fq$ is perfect, residually semi-simple
and does not satisfies any polynomial identity.
\end{enumerate} 
\end{cor3cases}

\noindent By \cite{Grab1}, any Krichever-Novikov algebra
$\Vect_X$ is simple, that is of type A. Given any point
 $P$ on the curve $X$, the Lie subalgebra $\Vect_X^{(2)}$ of vector fields vanishing twice at $P$ is residually nilpotent,
 that is of type B.  However, we do not know examples of Noetherian Lie algebras of type $C$.

It follows from Corollary 
\ref{3cases} that the Sierra-Walton conjecture can be reduced to these three types of Lie algebras. 

Hovewer Theorem A and some intuitive arguments, briefly outlined in Subsection \ref{undecidable}, suggest that the conjecture could be undecidable for some simple Lie algebras $\fS$.

\subsection{The ideas of the proofs}\label{idea}

The proof of Theorem A is quite long. Although its statement looks formal, its proof is based on concrete results.
The core of the proof consists of the following  result:

\begin{characteristic}
Let $\fG$ be a weakly Noetherian Lie algebra.
 
 Any maximal ideal of codimension $>1$ is characteristic.
\end{characteristic}

The difficult point in the proof of Theorem \ref{characteristic} concerns maximal ideals of finite codimension. It is based on
the theory of affine Lie algebras, as
initiated  in \cite{Tits1,Tits2},  and the next result:

\begin{structure1} (simplified version) Let $\fG$ be a  weakly Noetherian Lie algebra and let $d> 1$ be an integer.

Then $\fG$ contains only finitely many maximal ideals $\fm$  of codimension $d$.
\end{structure1}

Its proof combines the theory of central identities of Formanek \cite{Formanek} and Razmyslov \cite{Razmyslov}, together with  the finitness of the dimension of  centroids, i.e.  Proposition \ref{centroid}.

Obviously  the proof of Theorem A runs by transfinite induction, and we will illustrate the use of Theorem \ref{characteristic}
to go to the second step. Using hypotheses (h1) and (h2), it is quite easy to prove that 
$\fG_{(1)}$ is perfect and that $\fG_{(0)}/\fG_{(1)}$ is a central extension of $\fG_{(0)}/\rad(\fG_{(0)})$. To continue
by induction, we consider a maximal ideal $\fm$ of $\fG_{(1)}$
and the corresponding simple Lie algebra $\fS=\fG_{(1)}/\fm$.
Using  Theorem \ref{characteristic},  we show that
the Lie algebra $\Out(\fS)$ of outer derivations of $\fS$
is nonzero. Thus 
by  \cite{Bourbaki}, $\fS$ is infinite dimensional, which allows to pursue the induction. Moreover by
a theorem of  Grabowski
\cite{Grab2}, $\fS$ is not a 
Krichever-Novikov Lie algebra either.

\bigskip
The proof of Theorem B starts with the following theorem,
whose proof is essentially based on author's result
\cite{Mathieu-jalg,Mathieu-inv1,Mathieu-inv2, AndMath} for
algebraically closed fields. The descent
argument  from $\overline{K}$ to $K$ is an easy
consequence of Hilbert's 90 Theorem. 

\begin{simplequot} Let $\cL$ be a simple
infinite-dimensional $\Z^n$-graded
Lie algebra over $K$. 

If $\cL$ is strictly Noetherian, 
then $\cL$ is isomorphic to
$\Witt(L)$ or $\Vir(L)$ for some 
finite extension $L$ of $K$.
\end{simplequot}

Given an arbitrary weakly  Notherian Lie algebra $\fG$, the 
definition of the filtration

$$\fG\supset \fG_{(0)}\supset \fG_{(1)}\supset \ldots\fG_{(\alpha)}\supset\ldots$$ 

\noindent is generalized in Section \ref{proofA}. However, when
$\fG$ does not satisfy (h1) or (h2), we have
$\fG\neq \fG_{(0)}$. Nevertheless, $\fG_{(0)}$ satisfies the conclusion of Theorem A.

Now, let $\cL$ be a Lie algebra satisfying the hypotheses of Theorem B. Theo\-rem A(d) and Theorem \ref{simplequot} easily imply that $\cL_{(1)}=0$. 
Moreover, we also deduce
that $\cL_{(0)}/\rad(\cL_{(1)})$ is  a finite sum
of copies of Witt and Virasoro Lie algebras.
Using an argument about representation theory of
finite dimensional Lie algebras, 
we also prove that $\cL/\cL_{(0)}$ is finite dimensional.
Then Grabowski\'s theorem  \cite{Grab2} is used to show that, up to isogeny, $\cL$ is the direct sum of
$\cL/\cL_{(0)}$ and $\cL_{(0)}$, which concludes the proof of Theorem B.

\subsection{Organization of the paper}
Section \ref{Def} contains the main definitions and 
notations. In the next section \ref{counter},
we provide examples of non-Noetherian Lie algebras, 
in order to underline the specificities of the Noetherian or weakly Noetherian Lie algebras.

Then, Part A is devoted to the proof of Theorem A.
Section \ref{Noetherianity} explains the basic properties of weakly Noetherian Lie algebras.
In Section \ref{ExKN}, we show that the Krichever-Novikov Lie algebras are Noetherian. 
The next three Section prepare  the
proof of Theorem A. In Section
\ref{Seccentroid}, an easy argument shows  that centroids of weakly Noetherian Lie algebras are finite dimensional. In section \ref{Ass(c)} we prove a refined version of Theorem \ref{structure1}
and, in Section \ref{characteristic},
we prove Theorem \ref{characteristic}.
Finally, the proof of Theorem A
is given in the last section \ref{proofA}.

Next, Part B, provides the proof of Theorem B.
Section \ref{Z-Noetherian} provides the general properties of $\Z^n$-graded Noetherian Lie algebras.
In Section \ref{classification}, we explain the proof of 
Theorem \ref{simplequot}, and the last Section \ref{proofB} is devoted to the proof of Theorem B.

\section{Main Definitions and Notations}\label{Def}

From now on, we denote by $K$ a field of characteristic zero, and we write $\overline K$ for its algebraic closure.

Depending upon the context, we will 
use symbols like $\fG$, $\fg$, $\cL$ for Lie algebras
and $\fm$, $\fr\ldots $ for ideals of Lie algebras. Usually $\fG$ denotes a Lie algebra of 
arbitrary dimension, $\fg$ a Lie algebra of finite dimension and $\cL$ a graded Lie algebra.

We collect now the main definitions, which mostly concern the  infinite dimensional Lie algebras. 

\subsection{Simple Lie algebras and radicals}
As usual, a one-dimensional Lie algebra
is not considered as a simple Lie algebra.
Accordingly, the
{\it radical} of  a Lie algebra $\fG$, denoted by $\rad \fG$,
is  the intersection  of all  its maximal  ideals of codimension $> 1$. Its finite radical,
denoted as $\rad^f \fG$, is the intersection  of all  its maximal  ideals $\fm$ with

\centerline{$1<\dim \fm<\infty.$}

We define a {\it semisimple} Lie algebra as a {\it finite dimensional Lie algebra} which is a direct sum of simple Lie algebras.  With our restrictive convention, neither a simple Lie algebra of infinite dimension nor an infinite sum
of simple finite dimensional Lie algebra is semisimple.

\subsection{Derived and descending central series}
Let $\fG$ be a Lie algebra of arbitrary dimension.  Recall that $\fG$ is called {\it perfect} if $\fG=[\fG,\fG]$.

The descending central series of $\fG$ is 
defined by induction by
\begin{align*}
\Cs^1\fG &\coloneqq\fG, \\
\Cs^{k+1}\fG &\coloneqq [\fG, \Cs^k\fG]  
\noindent{\rm\,\,for\,\,} k \ge 1,\\
\Cs^{\omega}\fG &\coloneqq \cap_{k \geq 1} \Cs^k\fG 
\end{align*}

\noindent and the derived series  by 
\begin{align*}
\cD^{0} \fG &\coloneqq \fG\\
\cD^{k +1} \fG &\coloneqq [\cD^{k} \fG, \cD^{k} \fG]\noindent{\rm\,\,for\,\,} k \ge 0,\\
\cD^{\omega}\fG &\coloneqq \cap_{k \geq 1} \cD^k\fG.
\end{align*}

\noindent With these conventions we have

\centerline{
$[\Cs^k\fG,\Cs^\ell\fG]\subset \Cs^{k+\ell}\fG$ and 
$\cD^{k} \fG\subset \Cs^{2^k}\fG$.}

\medbreak
 We say that the derived series
$\cD^k \fG$  {\it stabilizes} if
there is an integer $N$ such that

$$\cD^k\fG=\cD^N\fG \,\,\forall k\geq N,$$

\noindent which is equivalent to
$\cD^{N+1}\fG=\cD^N\fG$. The statement that  the central 
series $\Cs^k\fG$ stabilizes is similarly defined. In general, the Lie algebra 
$\cD^{\omega}(\fG)$ is not perfect, and , for any ordinal $\alpha$ we could define $\cD^{\alpha}\fG$.
In the paper, we are only interested in the last term of the series, say $\cD^*\fG$, which can be simply defined without using ordinals.
Indeed $\cD^*\fG$ is  the sum of all perfect ideals, which is obviously
the largest perfect ideal in $\fG$.
When $\cD^*\fG=0$, the Lie algebra is called
{\it purely imperfect}.

We shall use repeatedly the following  obvious facts.

\begin{lemma}\label{stabilization}
Let $\fG$ be a Lie algebra. 
\begin{enumerate}
\item[(a)] If $\fG/[\fG,\fG]$ has finitedimension and its central series stabilizes, then $\fG / \Cs^{\omega}\fG$ is a finite dimensional  nilpotent Lie algebra.

\item[(b)] If the derived series of $\fG$ stabilizes, then 
$\cD^{\omega}\fG$ is a perfect Lie algebra.

\end{enumerate}
\end{lemma}

\subsection{Residual and strict properties}
Let $\cP$ be a certain property of Lie algebras.
When a Lie algebra $\fG$ satisfies $\cP$, we say
that $\fG$ is $\cP$. 
Given $\fG$  a Lie algebra, we say that

\begin{enumerate}
\item[(a)] $\fG$ is {\it residually $\cP$}
if, for any $x\in \fG$, there is an ideal
$\fm\subset \fG$ with $x\notin\fm$ such that
$\fG/\fm$ is $\cP$,

\item[(b)] $\fG$ is {\it strictly $\cP$}
if $\fG\otimes L$ is $\cP$ for any finite extension $K$ of $L$.
\end{enumerate}

For exemple,  $\fG$ is  {\it residually nilpotent} if $\Cs^\omega  \fG=0$. 
Furthermore,  $\fG$ is   {\it residually semi-simple} if  $\rad^f\g=0$.
Equivalently, there is a collection $\{\fs_\alpha\mid \alpha\in A\}$ of finite dimensional simple Lie algebras $\fs_\alpha$ such that
$\fG$ is a dense subalgebra of $\prod_{\alpha\in A}\,\fs_\alpha$.

The previous two properties are not exclusive each other. For example, the Lie $\fG=\fsl(2)\otimes tK[t]$. is simultaneously residually semi-simple and residually nilpotent.

\subsection{Other definitions}
We now recollect standard definitions for
arbitrary Lie algebras $\fG$.  

\begin{enumerate}
\item[(a)]  Given another Lie algebras $\fG'$,
we say that $\fG$ and $\fG'$ are {\it commensurable} if they contain the same   subalgebra of finite codimension. 

\item[(b)] A Lie algebra
$\fs$ is called a {\it section} of $\fG$ if there is a Lie subalgebra  $\fp$ of $\fg$ and an ideal
$\fq$ of $\fp$ such that 

\centerline{$\fs\simeq \fp/\fq$.}

\item[(c)] The {\it centralizer} of a subset $S\subset\fG$ is the Lie subalgebra
$$C_\fG(S)\coloneqq\{x\in \fG\mid [x,S]=0\}.$$

The {\it normalizer} of a subalgebra 
$\fa\subset\fG$ is the Lie subalgebra
$$N_\fG(\fa)\coloneqq\{x\in \fG
\mid [x,\fa]\subset\fa\}.$$

\item[(d)]
The  Lie algebra $\fG$ is called  {\it just-infinite} if $\dim\,\fG=\infty$ but any nonzero ideal
of $\fG$ has finite codimension.

\item[(e)] Given an arbitrary algebra $A$, we denote as $\Der\,A$ the Lie algebra of derivation of $A$.
For a Lie algebra $\fG$, the Lie algebra of {\it outer derivations} is

$$\Out(\fG):= \Der\fG/\ad(\fG).$$

\item[(f)] Given a Lie algebra $\fG$ and an ideal
$\fm$, we say that $\fG$ is an {\it extension} of
$\fG/\fm$ by $\fm$. When $\fm$ is a central ideal,
$\fG$ is called a {\it central extension of 
$\fG/\fm$}. 

\item[(g)]  For a perfect Lie algebra $\fG$,
we write $\widehat{\fG}$ for its  universal central extension.

\item[(h)]
An ideal $\fm$ is called a 
{\it characteristic ideal} if
$\partial \,\fm\subset\fm$ for any
derivation $\partial$ of $\fG$.
\end{enumerate}

Recall the following folklore result.

\begin{lemma}\label{perfect-char} Any perfect ideal $\fm$ is characteristic.
\end{lemma}

\begin{proof} For any derivation $\partial$
of $\fG$, we have

\centerline{
$\partial \,[\fm,\fm]=
[\partial\fm,\fm]\subset [\fG,\fm]\subset \fm,$}

\noindent which proves the claim.
\end{proof}

We will also use the following technical lemma:

\begin{lemma}\label{centralizer} Let $\fa$ be an ideal of a Lie algebra $\fG$
and  let $\fz$ be the center of $\fa$. Assume that
$\Out(\fa/\fz)=0$.
\begin{enumerate}
\item[(a)] If $\fz=0$, then as  Lie algebras
$$\fG= C_\fG(\fa)\oplus \fa\simeq  \fG/\fa\oplus \fa.$$
\item[(b)] If $\fG$ and $\fa$ are perfect, then as Lie algebras
$$\widehat{\fG}\simeq \widehat{\fG/\fa}\oplus \widehat{\fa}.$$
\end{enumerate}
\end{lemma}

\begin{proof} We first prove Assertion (a). Consider  the natural Lie algebra homorphism 

$$\pi:\fG\to\Der\fa ,  z\mapsto \ad(x)\mid_\fa.$$

\noindent By hypothesis we have $\Der\fa\simeq \fa$
and therefore $\fG/\Ker\pi\simeq \fa$. Since 
$\Ker\pi=C_\fG(\fa)$, we deduce

$$\fG= C_\fG(\fa)\oplus \fa\simeq  \fG/\fa\oplus \fa,$$

\noindent which proves the claim.

\medskip
We now prove Assertion (b). Let $\fz'$ be the space of all
elements $z\in \fa$ whose image in $\fa/\fz$ is central. Thus for
$z\in\fz'$, we have

$$\ad(x)\ad(y)(z)=0,\,\,\forall\, x,y\in\fa.$$

Since $\fa$ is perfect we have

$$[\fa,\fz']=[[\fa,\fa],\fz']\subset [\fa,[\fa],\fz']]=0,$$

\noindent therefore $\fz'=\fz$, which means that $\fa/\fz$ is centerless. We have just prove that
$\fG/\fz\simeq \fG/\fa \oplus \fa/\fz$, therefore we have

$$\widehat{\fG}\simeq \widehat{\fG/\fa}\oplus \widehat{\fa},$$

\noindent which completes the proof.
\end{proof}

\subsection{Structure theory of finite dimensional Lie algebra}

 The following classical result is attributed to E. Cartan, cf. 
 \cite{Bourbaki}{\S5}.

\begin{theorem} \label{cartan} Let $\fg$ be a finite dimensional Lie algebra.

\begin{enumerate}
\item[(a)]  The radical $\rad \fg$ of $\fg$ is the maximal solvable ideal.

\item[(b)] Set $\fs=\fg/\rad \fg$. Then
$\fg\simeq\fs\ltimes \rad \fg$.

\item[(c)] If $\fg$ is perfect, then $\rad \fg$ is nilpotent.

\item[(d)] If $\fg$ is solvable, then $[\fg,\fg]$ is nilpotent.

\item[(e)] If $\fg$ is semisimple, then $\Out(\fg)=0$.
\end{enumerate}
\end{theorem}

\section{Examples of non-Noetherian Lie algebras}\label{counter}

We provide examples of  non-Noetherian Lie algebras which
underline the specificity of Noetherian Lie algebras.
In what follows, $\fG$ denotes a Noetherian Lie algebra
and $\fH$ will designate a non-Noetherian Lie algebra.
For simplicity, we will always assume that $\fG$ and $\fH$ are perfect.

\subsection{Perfect Lie algebras without maximal ideals}

Theorem A states that   $\fG_{(\alpha)}=0$ for $\alpha$ big enough.
In fact, by Corollary \ref{idealACC}, any nonzero Noetherian Lie algebra admits a maximal ideal, therefore the descending series
$\fG_{(\alpha)}$ is decreasing.

However there are perfect Lie algebras $\fH$ without maximal ideals.
Such an example can be obtained easily using 
an HNN extension of free Lie algebras. Let
$\Free(\aleph_0)$ be a free Lie algebra over countably many letters. Since $\Free(\aleph_0)$ and $[\Free(\aleph_0),\Free(\aleph_0)]$ are isomorphic, there is an injective  Lie algebra homomorphism

$$\Phi: \Free(\aleph_0)\to \Free(\aleph_0)$$

\noindent such that $\Phi(\Free(\aleph_0))=
[\Free(\aleph_0),\Free(\aleph_0)]$. 

Let 
$\F_1, \F_2,\ldots$ be copies of the Lie algebra
$\Free(\aleph_0)$. Thus we get an inductive system

$$\F_1 \xlongrightarrow{\Phi}\F_2
 \xlongrightarrow{\Phi}\ldots.$$

\noindent Let

$$\cF=\varinjlim\,\F_i=\cup_{i=1}^\infty \F_i$$

\noindent be its inductive limit.

Since $\F_i=[\F_{i+1},\F_{i+1}]$, the Lie algebra
$\cF$ is perfect. It is clear that $[F_1,F_1]$ is an ideal. We observe that  $\fH:=\cF/[\F_1,\F_1]$ is an union of the
solvable ideals $\F_n/[\F_1,\F_1]=\F_n/\cD^n \F_n$, therefore
$\fH$ has no simple quotients. 

In conclusion, we  ask

\hskip2cm{\it Does a perfect  weakly Noetherian Lie algebra }

\hskip2cm{\it $\fG\neq 0$ admits at least one maximal ideal?}

\subsection{Perfect Lie algebras $\fH$ with a non-characteristic radical}

Theorem \ref{characteristic} insures that 
any maximal ideal of $\fG$ is characteristic.
In particular, $\rad(\fG)$ is characteristic.

However, it is not the case for non-Noetherian Lie algebra.
A typical example is 
$\fH=\fsl_2\otimes K[[t]]$. Its radical, namely
$\fsl_2\otimes t K[[t]]$,  is not stable by the
derivation $\partial=\frac{\d}{\d t}$.

\subsection{Simple quotients of the radicals}

For any weakly Noetherian Lie algebra, it will  be established that 
\begin{enumerate}
\item[(a)] all simple quotients of
$\rad^f(\fG)$ or of $\rad(\fG)$ have infinite dimension,
\item[(b)] no simple quotient of
$\rad(\fG)$  is a Novikov-Krichever-algebra
\end{enumerate}

However, none of these statements is correct for 
non-Noetherian  Lie algebras $\fH$.

\begin{enumerate}
\item[(a)] Let $\fs\subset \Der K[t]$ be the Lie algebra with
basis $\{\frac{\d}{\d t}, t\frac{\d}{\d t}, 
t^2\frac{\d}{\d t}\}$ and set
$\fH=\fs\ltimes \big(\fsl_2\otimes K[t]\big)$.

Then we have $\rad^f(\fH)=\rad(\fH)= \fsl_2\otimes K[t]$
and these radicals have  quotients isomorphic to 
 $\fsl_2$.

\item[(b)] Let $C$ be a smooth affine curve, let $S$ be a smooth
affine surface and let $\pi:S\to C$ be a smooth fibration.
By assumption, the fibers

$$S_c=\{x\in S\mid \pi(x)=c\}$$ 

\noindent are smooth curves for any $c\in C$.
Set 

$$\fH=\{\partial\in \Der K[S]\mid \partial K[C]\subset K[C]\}.$$

The kernel of the natural morphism 
$\fH\to\Vect_C, \partial\mapsto \partial\mid_{K[C]}$ is clearly 
the radical $\rad(\fH)$ of $\fH$. Moreover each
Krichever-Novikov Lie algebra $\Vect_{S_c}$ is an homorphic image of
$\rad(\fH)$.
\end{enumerate}

\part*{\hskip35mm Part A: Proof of Theorem A}

\section{Various Noetherianity properties for Lie algebras}\label{Noetherianity}

\subsection{Definition of Noetherianity for Lie algebras}
We define three notions of Noetherianity for Lie algebras:

\begin{enumerate}
\item[(1)]
A Lie algebra $\fG$ is called {\it weakly Noetherian}  if 
$$\dim\,\fm/[\fm,\fm]<\infty$$
\noindent for all subalgebras $\fm$ in $\fG$.
\item[(2)]
A weakly Noetherian Lie algebra $\fG$ is called
{\it Noetherian} if  satisfies the  ascending chain condition (ACC in brief)  on 
the characteristic ideals.
\item[(3)]
A Lie algebra $\fG$ is called {\it strongly Noetherian} if all its subalgebras are finitely generated. 
\end{enumerate}

We observe that the notion of Noetherianity for
the associative algebras is preserved by finite 
extensions, as shown by the following 
well-known statement.

\begin{lemma} If an associative algebra $A$
 is left Noetherian, then $A$ is strictly left  Noetherian.
\end{lemma}

\begin{proof} Let $L\supset K$ be a finite extension. Obviously, 
$A\otimes L$ is
a left free $A$-module of finite rank.
Thus $A\otimes L$ satisfies the ACC condition on left $A$-submodules. We deduce that $A\otimes L$ satisfies the ACC on left $A\otimes L$-modules, that is on the left ideals.
Therefore
$A\otimes L$ is left Noetherian.
\end{proof}

By contrast,  the notions of Noetherianity for the Lie algebras are  not obviously preserved by finite extensions. For clarification, let state the following obvious Corollary.

\begin{cor}\label{logical} The various notions of Noetherianity are logically connected by the following scheme:

\begin{align*}
&U(\fg) \text{\,\,strictly Noetherian}
&\Leftrightarrow \hskip17mm & U(\fg) \text{\,\, Noetherian}\\
&\hskip12mm \Downarrow&&\hskip12mm\Downarrow\\
&\fg \text{\,\,strictly strongly Noetherian}&
\Rightarrow\hskip17mm
&\fg \text{\,\, strongly Noetherian}\\
&\hskip12mm \Downarrow&&\hskip12mm\Downarrow\\
&\fg \text{\,\,strictly Noetherian}&
\Rightarrow\hskip17mm
&\fg \text{\,\, Noetherian}\\
&\hskip12mm\Downarrow&&\hskip12mm\Downarrow\\
&\fg \text{\,\,strictly weakly Noetherian}
&\Rightarrow\hskip17mm &\fg \text{\,\, weakly Noetherian}
\end{align*}
\end{cor}

\subsection{Elementary properties of weakly Noetherian Lie algebras}

We have collected the elementary properties
of weakly Noetherian Lie algebras in the
following Lemma.

\begin{lemma}\label{basic} 
Let $\fG$ be a weakly Noetherian Lie algebra. 

\medbreak
\begin{enumerate}[leftmargin=*,label=\rm{(\roman*)}]

\medbreak
\item[(a)]\label{item:section} Any section
$\fs$ of   $\fG$  is  also weakly Noetherian. 

\medbreak
\item[(b)]\label{item:abelian-section} Any abelian  section  of $\fG$ is finite dimensional.

\medbreak
\item[(c)]\label{item:nofree}
$\fG$ does not contain nonabelian free subalgebras.

\medbreak
\item[(d)]\label{item:commensurable}
If  $\fG'$ is commensurable with  $\fG$, then
$\fG'$ is also weakly Noetherian.

\medbreak
\item[(e)]\label{item:derived-finite-codim}
The $k$-th term $\cD^{k} \fG$ of the derived series has finite codimension
in $\g$ for any $k \geq 1$. 
In particular, if $\fG$ is solvable, then
then $\dim \fG < \infty$. 

\medbreak
\item[(f)]\label{item:derived-virnilpotent}
We have $\cD^\omega \fG=\Cs^\omega [\fG,\fG]$. 

\medbreak
\item[(g)] \label{item:perfect-virnilp}
If $\fG$ is perfect and $\fm$ is an ideal of  finite codimension, 
then $\cD^\omega\fm= \Cs^\omega\fm$.
\end{enumerate}
\end{lemma}

\noindent
\pf

Assertions (a) and (b) are obvious. 

We prove Assertion (c) by contradiction.
Otherwise, $\fG$ contains a free Lie algebra
over two generators $x$ and $y$. The
subalgebra $F$ generated by the set 
$S:=\{\ad(x)^n(y)\mid n\geq 0\}$ is 
the free Lie algebra over $S$. Therefore
$F/[F,F]$ is an infinite dimensional 
abelian section,
which contradicts weak Noetherianity.

Assertion (d) amounts to prove that if
 $\fG'$ contains $\fG$ as a subalgebra 
 of finite codimension, then 
 $\fG'$ is also weakly Noetherian.
 Let $\fa'$ be an arbitrary subalgebra
 of $\fG'$.
 Set $\fa=\fa'\cap \fG$. In view of the short exact sequence
 
 $$0\to \fa/[\fa,\fa]\to\fa'/[\fa,\fa]\to
 \fa'/\fa\to 0,$$

\noindent we deduce that 

\begin{align*}
\dim \fa'/[\fa',\fa']&\leq\dim \fa'/[\fa,\fa]\\
&=\dim \fa/[\fa,\fa]+\dim \fa'/\fa\\
&\leq \dim \fa/[\fa,\fa]+\dim \fG'/\fG<\infty,
\end{align*}

\noindent which proves that $\fG'$ is weakly Noetherian.

Assertion (e) follows from the fact that
$\fG/\cD^k\fG$ has a composition series consisting of $k$ abelian sections, thus
$\fG/\cD^k\fG$ has finite dimension.

We now prove Assertion (f). 
Let $k$ be any integer. It has been proved that
the solvable Lie algebra $\fG/\cD^k\fG$ is 
finite dimensional. By Theorem \ref{cartan}(d),
$[\fG/\cD^k\fG,\fG/\cD^k\fG]$ is nilpotent.
Therefore 

$$\cD^k\fG\supset \cC^\omega [\fG,\fG],$$

\noindent thus
$\cD^\omega\fG\supset \cC^\omega [\fG,\fG]$.
Since obviously

$$\cD^\omega \fG=\cD^\omega[\fG,\fG]\subset\Cs^\omega [\fG,\fG], $$
we finally deduce that

$$\cD^\omega\fG=\Cs^\omega [\fG,\fG].$$

Finally, we prove Assertion (g). 
Set $\fg=\fG/\cD^k \fm$. By the previous 
consi\-de\-rations, $\fg$ is 
of finite dimensional. Since $\fm/\cD^k\fm$ is a solvable ideal of the perfect Lie algebra $\fg$, Theorem \ref{cartan}(a)(c) asserts that  
$\fm/\cD^k\fm$ is nilpotent. Therefore
$$\cD^k\fm\supset \cC^\omega\fm,$$

\noindent for all integers $k$. It follows that
$\cD^\omega\fm\supset \cC^\omega\fm$ and therefore

$$\cD^\omega\fm= \cC^\omega\fm,$$ 

\noindent as claimed.
\end{proof}

\subsection{Finite closure of ideals}

Let  $\fG$ be a Lie algebra, and let $\fm$ be an ideal. Let $\cF$ be the set of all
ideals $\fr\supset\fm$ such that 
$\dim\fr/\fm<\infty$.
We define the {\it finite closure} 
$\Cl_{\fG}(\fm)$ of $\fm$
in $\fG$ as the ideal

$$\Cl_{\fG}(\fm)=\sum_{\fr\in\cF}\,\fr.$$

\noindent For an ideal $\fm$ of finite codimension, we have $\Cl_{\fG}(\fm)=\fG$,
therefore this notion 
is only interesting for ideals
of infinite codimension.
An ideal $\fm$  is called 
{\it closed} if $\fm=\Cl_{\fG}(\fm)$.
In general, $\cF$ has no maximal 
element, thus  $\Cl_{\fG}(\fm)/\fm$ could be infinite dimensional.

\begin{lemma}\label{closure} 
Any chain of finite dimensional ideals 
in  a weakly Noetherian Lie algebra $\fG$ stabilizes.

Moreover,  any ideal $\fm$ satisfies

\begin{enumerate}
\item[(a)] $\Cl_{\fG}(\fm)/\fm$ is finite 
dimensional, and
\item[(b)] $\Cl_{\fG}(\fm)$ is a characteristic ideal.
\end{enumerate}
\end{lemma}

\begin{proof} We first prove the first assertion
by contradiction. Assume otherwise, thus there 
is  an increasing chain

$$0=\fm_0\subset \fm_1\subset\fm_2\subset\ldots$$

\noindent of finite dimensional ideals. We can 
select a subchain in a way that

$$\dim \fm_{k}>\dim \fm_{k-1}^2.$$

Since $C_{\fm_{k}}(\fm_{k-1})$ is the kernel of the Lie algebra morphism

$$\ad:\fm_{k}\to \fgl(\fm_{k-1}), 
x\mapsto \ad(x)\mid_{\fm_{k-1}},$$

\noindent  dimension considerations show that there is an element
$x_k\in C_{\fm_{k}}(\fm_{k-1})$ which 
belongs to $\fm_{k}\setminus \fm_{k-1}$.
By construction, the elements $x_1,x_2,\cdots$
are linearly independent and they pairwise commute, which contradicts that $\fG$ is weakly Noetherian.

We now prove Assertion (a). Since any chain
of finite dimension ideals in $\fG/\fm$ stabilizes,
we conclude $\Cl_{\fG/\fm}(0)$ is finite dimensional.
Since 
$$\Cl_{\fG}(\fm)/\fm\simeq\Cl_{\fG/\fm}(0),$$ 

\noindent we conclude that  $\dim \Cl_{\fG}(\fm)/\fm<\infty.$

In order to prove Assertion (b), we can assume that
$\fm$ is closed. Let
$\partial$ be a derivation of $\fg$
and set $\fm'=\fm+\partial\fm$.
We have 

$$\partial[\fm,\fm]\subset 
[\partial\fm,\fm]\subset [\fg,\fm]\subset\fm.$$

\noindent It follows that 

$$\dim\,\fm'/\fm\leq \dim \fm/[\fm,\fm]<\infty.$$

\noindent Since $\fm'$ is an ideal and $\fm$ is closed, we have 
$\fm'=\fm$ which amounts to
$\partial\,\fm\subset\fm$. Since $\fm$ is stable by any derivation $\partial$, we conclude that
$\fm$ is characteristic.
\end{proof}

\begin{cor}\label{idealACC} Let $\fG$ be a Noetherian Lie algebra and let $\fa$ be an ideal.
\begin{itemize}
\item[(a)]  $\fa$  satisfies the ACC condition on chains of ideals,
\item[(b)] In particular, any quotient of $\fa$ is a Noetherian Lie algebra.
\end{itemize}
\end{cor}

\begin{proof} Let

$$\fm_1\subset\fm_2\subset\ldots$$

\noindent be a chain of ideals of $\fa$. By Assertion (b) of Lemma \ref{closure}, 
 $\Cl_\fa(\fm_k)$ is an
ideal of $\fG$, for any integer $k$.
By Assertion (b) of Lemma \ref{closure}
the chain 

$$\Cl_\fG( \Cl_\fa(\fm_1))
\subset\Cl_\fG( \Cl_\fa(\fm_2))\subset\ldots.$$

\noindent consist of characteristic ideals of $\fG$. Hence the chain stabilizes,
so there exist an integer $N$ such that

$$\Cl_\fG( \Cl_\fa(\fm_k))=\Cl_\fG( \Cl_\fa(\fm_N)),$$

\noindent for any $k\geq N$. Hence the chain

$$\fm_k\subset\fm_{k+1}\subset
\fm_{k+2}\ldots$$

\noindent is nested between $\fm_k$ and
$\Cl_\fG( \Cl_\fa(\fm_k))$. By Assertion (a)
of Lemma \ref{closure}, $\fm_k$ has finite codimension in
$\Cl_\fG( \Cl_\fa(\fm_k))$. We deduce that the
chain $\fm_k\subset\fm_{k+1}\subset
\fm_{k+2}\ldots$ stabilizes, which proves that $\fa$ satisfies the ACC condition on ideals,
that is Assertion (a).

It follows that any quotient of $\fa$ is Noetherian.
\end{proof}

\begin{cor}\label{ideal-just} Let $\fG$ be a Noetherian Lie algebra. 

If $\fG$ is just infinite,
then any nonzero ideal of $\fG$ is just-infinite.
\end{cor}

\begin{proof} Let $\fa$ be a nonzero ideal
of $\fG$ and let $\fb$ be a nonzero  ideal
of $\fa$. By Lemma \ref{closure}(b),
$\Cl_{\fa}(\fb)$ is an ideal of 
$\fG$, thus by hypothesis

$$\dim \fG/\Cl_{\fa}(\fb)<\infty.$$

\noindent Moreover by Lemma \ref{closure}(a)
we have

$$\dim \Cl_{\fa}(\fb)/\fb<\infty,$$

\noindent from which it follows that 
$\fb$ has finite codimension in $\fa$. Hence
$\fa$ is just-infinite.
\end{proof}

\section{Examples of strongly Noetherian Lie algebras.}\label{ExKN}

\noindent
The {\it Krichever-Novikov algebras}
are the Lie algebras $\Vect_X$ of vector fields on a smooth affine curve $X$.  However, unlike
 \cite{KN}, we do not assume that 
 $X$ has exactly two points at infinity. We show that the Krichever-Novikov algebras, which are simple by Grabowski Theorem \cite{Grab1}, are Noetherian
 and just-infinite.

Throughout the section, we will use the following notation. Given a graded Lie algebra 
$\fG=\oplus_{n\in \Z} \fG_n$, we set

$$\fG_{\geq 1}=\oplus_{n\geq 1} \fG_n.$$

\subsection{Noetherianity for Filtered Lie algebras}

A Lie algebra $\g$ endowed  with a filtration
$0=\g(-1)\subset\fg(0)\subset \fg(1)\subset\cdots$ satisfying

\begin{align*}
[\fg(n),\fg(m)]&\subset \fg(n+m) & \cup_n\g(n)&=\g
\end{align*}

\noindent is called a {\it filtered  Lie algebra.}
The associated graded Lie algebra is

$$\fG=\oplus_{n\geq  0}\,\,\fG_n$$

\noindent where $\fG_n=\fg(n)/\fg(n-1)$.

\noindent Here {\it we do not assume that the homogenous component $\fG_0=\fg(0)$  has finite dimension.} However,  the hypotheses of the lemma will
imply that,  for $n\geq 1$, the components 
$\fG_{n}$ are finite dimensional.

\begin{lemma}\label{Noeth-just}
 Let $\g$ be a filtered Lie algebra, with associated graded algebra $\fG$. If

 \begin{itemize}
\item[(a1)]The Lie algebra $\fG_{\geq 1}$
is weakly Noetherian, and

\item[(a2)] the Lie algebra 
$\fg(0)$ is strongly Noetherian,
\end{itemize}

\noindent then the Lie algebra $\g$ is strongly  Noetherian.

\medbreak
Moreover, if 

 \begin{itemize}
\item[(b1)] All  graded ideals   of any graded subalgebra $\fB$ of $\G_{\geq 1}$ have 
finite codimension in $\fB$, and

\item[(b2)] any subalgebra of  $\fg(0)$ is just-infinite,
\end{itemize}

\noindent then all subalgebras of $\g$ are  just-infinite.
\end{lemma}

\begin{proof}
Any subalgebra $\fb$ of $\fg$ admits a filtration

$$0=\fb(-1)\subset\fb(0)\subset \fb(1)\subset\ldots$$ 

\noindent defined by $\fb(n)=\fb\cap \fg(n)$ for all
$n\geq 1$. Let $\fB$ be the associated graded 
Lie algebra.

We now prove the first assertion, namely
that $\fb$ is finitely generated. 
Since $\fB$ is a Lie subalgebra of $\fG$, Hypothesis (a1)
implies that

$$\dim \fB_{\geq 1}/[\fB_{\geq 1},\fB_{\geq 1}]
<\infty.$$

Since $\fB_{\geq 1}$ is positively graded, we deduce that $\fB_{\geq 1}$ is generated by a finite set
$y_1,\dots,y_m$ of homogenous elements of degree respectively
$d_1,\ldots,d_m$. Hence, there are  elements $x_1,\cdots,x_m\in\fb$ such that

\centerline{$x_i\in\fb(d_i)$ and $y_i=x_i\mod \fb(d_i-1)$
for all $i=1,\ldots,m$.}
Let $\fb'$ be the subalgebra generated by  $x_1,\cdots,x_m$.  It is clear that, as  a vector space we have

$$\fb=\fb'+\fb(0).$$ 

\noindent Both subalgebras  $\fb'$ and $\fb(0)$ are   finitely generated, therefore
$\fb$ is finitely generated, which proves the first assertion.

We now turn to the second assertion, that is that $\dim\fb/\fm<\infty$ for any nonzero ideal
$\fm$ of $\fb$. Let 

$$0=\fm(-1)\subset\fm(0)\subset \fm(1)\subset\ldots$$ 

\noindent be the filtration of $\fm$ previously defined and let 
$\fM$ be the associated graded Lie algebra.

It is obvious that $\fm(0)$ is an ideal of  $\fb(0)$ and it is clear that $\fM_{\geq 1}$ is an ideal
of $\fB_{\geq 1}$. Therefore
$\fb(0)/\fm(0)$ and $\fB_{\geq 1}/\fM_{\geq 1}$ are finite dimensional by hypotheses. We deduce that $\fb/\fm$  is finite dimensional, which proves the second assertion.
 \epf

\subsection{The Novikov-Krichever algebras are strongly Noetherian}

\begin{lemma}\label{graded-just}

Let $\fB$ be a graded subalgebra of
$\Witt_{\geq 1}(K)$.
\begin{itemize}
\item[(a)]
The Lie algebra  $\fB$  is finitely generated.
\item[(b)] Any nonzero graded ideal $\fm$ of $\fB$  has 
finite codimension in $\fB$.
\end{itemize}
\end{lemma}

\pf Recall that $\Witt_{\geq 1}(K)$ has basis 
$\{L_n:=t^{n+1}\d/\d t\mid n\geq 1\}$ and Lie bracket

\begin{align*}
[L_n,L_m]=(m-n) L_{n+m}
\end{align*}

\noindent
 We observe that $\ad(L_n)(L_m)\neq 0$, except if $n=m$.
 Hence, for any $n\geq 1$,  $\Witt_{\geq 1}$ is a $K[\ad(L_n)]$-module of finite type.

We can assume $\fB\neq 0$, thus  $\fB$  contains some basis element $L_n$. Hence
$\fB$  is a  $K[\ad(L_n)]$-module of finite type.  Therefore  $\fB$ is finitely generated, which completes the proof of Assertion (a).

Similarly,  $\fm$ 
contains  an homogenous element $L_m$. We have just proved that $\fB/[L_m,\fB]$ has finite dimension, therefore $\fB/\fm$ is finite dimensional.
\end{proof}

\begin{theorem} Let $X$ be a smooth affine curve. 
The Krichever-Novikov algebra $\Vect_X$ is strongly Noetherian. 

Moreover any Lie subalgebra is just infinite. 
\end{theorem}

\begin{proof}
 The proof of both assertions runs by induction on the number of point at infinity of $X$. 
 Since for any field extension $L$ of $K$: 

$$ \Vect_X\otimes L {\rm\,\, Noetherian}
\Rightarrow \Vect_X{\rm\,\,Noetherian},$$

\noindent we can assume that all these points are rational.
 
If $X$ has no point at infinity, then 
$X$ is a projective curve of genus $g$. Thus
$\Vect_X$ is finite dimensional. Indeed
it is three-dimensional if $g=0$, one dimensional if
$g=1$ and $\Vect_X=0$ otherwise. 
 
Since $X$ is affine,  it admits
a point $P$ at infinity. Set $Y=X\cup\{P\}$ and $\fg=\Vect_X$.  For $n\geq 0$, let $\fg(n)$ the space of vector fields with a pole of order at most $n-1$
at $P$ and let $\fG$ be the graded space associated with the filtration 

$$\fg(0)\subset \fg(1)\subset\cdots.$$

Let $t$ be a local parameter at $P$ and let  
$\xi,\, \eta$ be vector fields with poles of order $n-1$ and $m-1$ at $P$. Up to some scalar multiple, we have

\begin{align*}
\xi=& t^{1-n}\d/\d t + o(t^{1-n})\\ 
\eta=& t^{1-m}\d/\d t + o(t^{1-m}).
\end{align*}

\noindent  Hence we have

\begin{align*}
[\xi,\eta]= (n-m) t^{1-m-n}\d/\d t + o(t^{1-m-n})
\end{align*}

It follows that 
$$[\fg(n),\fg(m)]\subset \fg(n+m),$$

\noindent and
$\fG_{\geq 1}$ carries a structure of Lie algebra isomorphic to a subalgebra of 
$\Witt_{\geq 1}(K)$. In fact if $Y$ has at least 
a point at infinity, $\fG_{\geq 1}$ is exactly
$\Witt_{\geq 1}(K)$. In the opposite, if $Y$ is projective, the set of degree of symbols has finitely many gaps, as the Riemann-Roch formula 
shows.

By definition, $\fg(0)$ is the Lie algebra of
vector fields on $Y$ which vanishes at $P$,
thus, by induction hypothesis,  $\fg(0)$ is strongly Noetherian. Together with Assertion (a) of Lemma \ref{graded-just}, $\Vect_X$ satisfies the Conditions (a1) and (a2) of
Lemma \ref{Noeth-just}. Thus $\Vect_X$ is strongly Noetherian.

The proof that any subalgebra of $\Vect_X$ is just infinite is almost identical, except that
is uses Assertion (b) of Lemma \ref{graded-just}
to check that $\Vect_X$ satisfies the Conditions (b1) and (b2) of
Lemma \ref{Noeth-just}.
\end{proof}

\subsection{Grabowski's Theorem}

The following statement is a particular case of a 
more general result proved by Grabowski. 

\begin{theorem}\label{Grabowski}\cite{Grab1}\cite{Grab2} Let $X$ be a smooth affine curve.
Then
\begin{enumerate}
\item[(a)] $\Vect_X$ is a simple Lie algebra, and
\item[(b)] $\Out(\Vect_X)=0$. 
\end{enumerate}
\end{theorem}

\section{Centroids of weakly Noetherian Lie algebras}\label{Seccentroid}

We now prove that the centroids of weakly
Noetherian Lie algebras are finite dimensional, see Proposition \ref{centroid}.
This result will be used in the next section.

Recall that the  {\it centroid} of a Lie algebra $\fg$, denoted as $\Cent(\fG)$,  is the algebra

$$\Cent(\fG):=
\{\theta\in\End(\fG)\mid \theta([x,y])=[x,\theta(y)]\,\forall\,x,y\in\fG\}.$$

\noindent For simplicity, we will simply denote as $\theta\,x$ the action of an operator
$\theta\in \Cent(\fG)$ on an element $x\in\fG$. 

\begin{lemma}\label{com} For any 
$\theta,\theta'\in \Cent(\fG)$, we have

\centerline {$\theta\theta'\,z=
\theta'\theta\,z\,\, \,\forall z\in 
[\fG,\fG]$.}
\end{lemma}

\begin{proof} By skew symmetry of the bracket,
we have 

$$\theta\,[x,y]=[\theta\,x,y]=
[x,\theta\,y],$$

\noindent for any $\theta\in \Cent(\fg)$ and
$x,y\in\fg$. It follows that

$$\theta'\theta\,[x,y]=
[x,\theta'\theta\,y]=
[\theta\theta'\,x,y]=
\theta\theta'\,[x,y],$$

\noindent which proves that 
$\theta\theta'\,z=
\theta'\theta\,z \,\forall z\in 
[\fG,\fG]$.
\end{proof}

For a subset $S\subset \fG$, the left ideal

$$\Ann_{\Cent(\fG)}(S):=
\{a\in \Cent(\fG)\mid aS=0\}$$

\noindent is called the {\it annihilator
in $\Cent(\fG)$ of S}.

\begin{lemma}\label{cent1} Let $\fG$ be a weakly Noetherian Lie algebra. The annihilator

$$\Ann_{\Cent(\fG)}([\fG,\fG])$$ 

\noindent is a  two-sided ideal of $\Cent(\fG)$
of finite dimension.
\end{lemma}

\begin{proof} The left ideal $\Ann_{\Cent(\fG)}
([\fG,\fG])$ is also
a right ideal by Lemma \ref{com}.

Let $\fz$ be the center of
$\fG$. 
For $a\in \Ann_{\Cent(\fG)}([\fG,\fG])$, we have

$$[a\fG,\fG]=a[\fG,\fG]=0,$$

\noindent hence $a\fG$ lies in $\fz$. It follows easily that 

$$\Ann_{\Cent(\fG)}([\fG,\fG])=\Hom(\fG/[\fG,\fG],\fz).$$ 

Since $\fG$ is
weakly Noetherian, its abelian sections $\fG/[\fG,\fG]$ and $\fz$
are finite dimensional, thus the ideal 
$\Ann_{\Cent(\fG)}([\fG,\fG])$ is finite dimensional, which completes the proof.
\end{proof}

 We now state our last preliminary result.

\begin{lemma}\label{cent2} Let $\fG$ be a weakly Noetherian Lie algebra, and let $I\subset \Cent(\fG)$ be a linear subspace.

If $I^2=0$, then $I$ is finite dimensional.
\end{lemma}

\begin{proof} Set $\fa= I\fG$.We have

$$[\fa,\fa]=[I\fG, I\fG]=I^2[\fG,\fG]=0,$$

\noindent hence the commutative ideal $\fa$ is finite dimensional. Therefore its commutant 

$$\fm:=C_\fG(\fa),$$

\noindent has finite codimension. Since
$\fm/[\fG,\fm]$ is an abelian section,
the ideal $[\fG,\fm]$ also have finite codimension.

We have 

$$I[\fG,\fm]=[I\fG,\fm]=[\fa,\fm]=0$$

\noindent Hence there is a natural embedding
$I\subset\Hom( \fG/[\fG,\fm],\fa)$, which proves that $I$ is finite dimensional.
\end{proof}

\begin{prop}\label{centroid} The centroid
$\Cent(\fG)$ of a 
 weakly Noetherian Lie algebra
is finite dimensional. 
\end{prop}

\begin{proof} Assume otherwise, that is
$\dim \Cent(\fG)=\infty$.
By lemma \ref{cent1}, the algebra

$$R:=\Cent(\fG)/\Ann_{\Cent(\fG)}([\fG,\fG]).$$

\noindent  is infinite dimensional.

We now find,
by induction,  an infinite
sequence $z_1,z_2\cdots$ of elements
in  $[\fg,\fg]$  such that the 
ideals 
$I_n:=\cap_{i=1}^n \Ann_{\Cent(\fG)}(z_i)$ satisfy 

$$I_n^2 z_{n+1}\neq 0.$$

To start the induction, we observe that  $R\neq 0$, hence there exists a nonzero element $z_1\in [\fG,\fG]$. Since $I_0=R$, we have $I_0^2 z_1\neq 0$.

Next, assume now  by induction that we have already found the elements $z_1,z_2,\ldots, z_n$.
Since for $z\in [\fG,\fG]$, the space  $R z$ is an abelian Lie subalgebra of $\fG$, the
ideal $\Ann_{\Cent(\fG)}(z)$ has finite codimension. Therefore 

$$I_n:=\cap_{i=1}^n \Ann_{\Cent(\fG)}(z_i)$$

\noindent has infinite dimension. By 
Lemma \ref{cent2} we have $I^2_n\neq 0$, and therefore
we can find $z_{n+1}\in[\fG,\fG]$ such that
$I^2_n\,z_{n+1}\neq 0$. Thus we have established the existence of the sequence $z_1,z_2\cdots$.

By Lemma \ref{com}, the algebra $R$ is commutative, so $I^2_n$ is the linear span
of the set $\{a^2\mid a\in I_n\}$. Therefore there exists elements 
$a_1\in I_0, a_2\in I_1,\ldots$ such that

$$a_n^2 z_n\neq 0, \forall n\geq 1.$$ 

Set $y_n=a_n z_n$. To conclude the proof, it is enough to show that the Lie algebra $\fa$ generated by the 
set $\{y_n\mid n\geq 1\}$  is an infinite dimension abelian Lie algebra. 

Since
$R$ is commutative, we observe that
for $n<m$, we have $a_m a_n z_n= a_n a_m z_n=0$.
Therefore

$$[y_n,y_m]=[a_n x_n, a_m x_m]=
[a_m a_n x_n, x_m]=0,$$

\noindent which proves that $\fa$ is abelian. It remains to prove that the elements
$y_n$ are linearly independant. Let

$$Y:=\sum_{i=1}^n\, c_i y_i$$

\noindent be any linear combination, where
$c_i$ are scalars and $c_n\neq 0$.
As before, we have

$$a_n y_i=a_na_i x_i=a_ia_nx_i=0$$ 

\noindent for $i<n$.
Thus $a_n.Y =c_n a_n.y_n=c_n a_n^2.z_n\neq 0$,
which shows that the set $\{y_i\mid i\geq 1\}$ is linearly independent.
\end{proof}

\begin{cor}\label{simplecentroid} Let $\fS$ be a weakly Noetherian simple 
Lie algebra. Then
\begin{enumerate}
\item[(a)] Its centroid $L$ is a finite field extension of $K$, hence $L\subset\overline{K}$, and
\item[(b)] $\fS\otimes_L\overline{K}$ is a simple
Lie algebra.
\end{enumerate}
\end{cor}

\begin{proof} The centroid $L$ of $\fS$ is obviously a division algebra. Since $\fS=[\fS,\fS]$, the algebra $L$ is commutative by Lemma \ref{com}. 
By Proposition \ref{centroid}, $L$ is a finite extension of $K$, which proves Assertion (a).

Assertion (b) follows from Jacobson density Theorem \cite{Jac} applied to the simple $U(\fS)$-module $\fS$.

\end{proof}

\section{Finite dimensional simple quotients}\label{Ass(c)}

Let $\fG$ be a weakly noetherian Lie algebra
and let $d>1$ be an integer. 

  We will prove that a refined version of the 
theorem \ref{structure1} stated in the introduction, namely 
that $\fG$ admits only finitely many maximal ideals 
$\fm$ such that $\fG/\fm$ is a simple Lie algebra
of dimension $d$ over its centroid 
$L:=\Cent(\fG/\fm)$. 

Since by  the previous corollary \ref{simplecentroid}
$\Cent(\fG/\fm)$ is a finite field extension 
of $K$, it means  that

$$\dim\,\fG/ \,\fm= d [L:K].$$

The proof is based on the theory of 
Formanek-Razmyslov Theorem about central identities
\cite{Formanek, Razmyslov}.

\subsection{Generalities about simple quotients}

We recall two elementary  well-known lemmas.

\begin{lemma}\label{dense} Let $R$ be a unital associative
algebra, and let $M, M_1,\cdots, M_n$ be a finite
collection of distinct maximal two-sided ideals.
Then the natural map

 $$M_1\cap M_2\cap\cdots\cap M_n\to  R/M$$
 
\noindent is surjective.
\end{lemma} 

\begin{proof} Obviously we have

\begin{align*}
&M+\big(M_1\cap \cdots \cap M_n\big)\\
\supset& M+M_1M_2\cdots M_n\\
\supset&(M+M_1)(M+M_2)\cdots (M+M_n)
\end{align*}

\noindent However, by  maximality assumption,
we have $M+M_i=R$, for all positive integer $i\leq n$.
Therefore $(M+M_1)(M+M_2)\cdots (M+M_n)=R$.
It follows that

$$M+\big(M_1\cap \cdots \cap M_n\big)=R,$$

\noindent which amounts to the surjectivity of

$$M_1\cap M_2\cap\cdots\cap M_n\to  R/M.$$
\end{proof}

For an integer $d\geq 1$ and a field $L$,
 let $\Mat_d(L)$ be the algebra of
$d\times d$ matrices over $L$.

\begin{lemma}\label{Wedderburn} Let $\fg$ be a simple  Lie algebra
of finite dimension $d$ over its centroid  $L:=\Cent(\fg)$. 

Then, the image of $U(\fg)$ in $\End(\fg)$ is isomorphic to $\Mat_d(L)$.
\end{lemma}

\begin{proof} The assertion is a consequence of Wedderburn's Theorem applied to the 
simple finite dimensional $U(\fg)$-module
$\fg$. 
\end{proof}

\subsection{The Formanek-Razmyslov Theorem}

We will now state an important result of Formanek and Razmyslov. Let $\Q<x_1,\ldots,x_m>$ be the 
free $\Q$-algebra  over the $m$-uple
 ${\bf x}=(x_1,\ldots,x_m)$ of noncommutative variables. Elements $P({\bf x})$ in 
$\Q<x_1,\ldots,x_m>$ are called 
{\it polynomials over the noncommutative variables}
$x_1,\ldots,x_m$. 

Recall that, for any integer $d$, the center of $\Mat_d(K)$ is $K$.

\begin{theorem}\cite{Formanek}\cite{Razmyslov}\label{RazFor}
Let $d$ be a positive integer. There is an integer $m$ and a homogenous nonconstant polynomial
$P_d({\bf x})\in \Q<x_1,\ldots,x_m>$ such that

\begin{itemize}
\item[(a)] For any ${\bf a}=(a_1,\cdots ,a_m)\in \Mat_d(K)^m$, 
$P_d({\bf a})$ is central, and

\item[(b)] $P_d({\bf a})=1$ 
for some ${\bf a}\in \Mat_d(K)^m$.
\end{itemize}
\end{theorem}

A polynomial $P_d$ satisfying the properties of the previous theorem is 
called a {\it central identity of degree $d$}.
For example $P_2(x_1,x_2):=[x_1,x_2]^2$ is a central identitity of degree $2$.
However,  there is no simple expressions
for a central identity of degree $d\geq 3$.  Indeed the proposed polynomials in the papers 
\cite{Formanek} and \cite{Razmyslov} are distinct.
For a  nice and recent account of central identities,
see \cite{BMRV}.

\subsection{Finite dimensional simple quotients}

We now prove the main result of the section.

\begin{theorem}\label{structure1} Let $\fG$ be a  weakly Noetherian Lie algebra and let $d> 1$.

Then $\fG$ contains only finitely many maximal ideals $\fm$ such that $\fG/\fm$ is a simple Lie algebra of dimension $d$ over its centroid $L=\Cent(\fG/\fm)$.
\end{theorem}

\begin{proof} Suppose otherwise. We will show that it contradicts Proposition \ref{centroid}

Thus there is
an infinite sequence 
$$\fm_1,\fm_2\ldots$$ 
of distinct ideals such that 
$\fg/\fm_i$ is a simple Lie algebra of dimension $d$ over its centroid 
$L_i:=\Cent(\fG/\fm_i)$.

Without loss of generality, we can assume that
$\cap_{i=1}^\infty \,\fm_i=0$, therefore there is an embedding

$$\fG\subset \prod_{i=1}^\infty \,\fG/\fm_i.$$

Let $P_d(x_1,\cdots,x_m)$ be a central identity, which does exist by Theorem \ref{RazFor}.
For any $m$-uple 
${\bf a}=(a_1,\ldots,a_m)\in U(\fG)^m$,
the element $\Ad(P_d({\bf a}))$ lies in 
$\Cent(\fG)$. Since 

$$\Cent(\fg)\subset 
\prod_{i=1}^\infty \,L_i,$$

\noindent $\Ad(P_d({\bf a}))$ is represented as an infinite sequence of scalars
$(x_1,x_2,\ldots)$, where each $x_i$ belongs to $L_i$.

Set $I_n=\cap_{i=1}^n\,\Ann_{U(\fG)}(\fG/\fm_i)$. By Lemmas \ref{dense} and \ref{Wedderburn}, there is 
${\bf a}^{(n)}\in I_{n-1}$ such that
$\Ad(P_d({\bf a}^{(n)}))$ acts on
$\fG/\fm_n$ as the identity. Therefore we have

\begin{align*}
\Ad(P_d({\bf a}^{(1)}))&=(1,x_2, \ldots)\\
\Ad(P_d({\bf a}^{(2)}))&=(0,1,y_3,\ldots)\\
\Ad(P_d({\bf a}^{(3)}))&=(0,0,1,z_4,\ldots)\\
&\ldots
\end{align*}

\noindent where $x_2,\ldots, y_3,\ldots, z_4,\ldots$ are some undetermined scalars.
Since 

$$\Ad(P_d({\bf a}^{(1)})), \Ad(P_d({\bf a}^{(2)})), \Ad(P_d({\bf a}^{(3)})),\ldots $$

\noindent are linearly independent, we conclude that $\Cent(U\fg)$ is
infinite dimensional, which contradicts Proposition
\ref{centroid}.
\end{proof}

\subsection{Lie Polynomial Identities}

A {\it Lie polynomial} 
$P(x_1,\ldots,x_m)$ is  an element
of the free Lie $K$-algebra over the Lie
variables $x_1,\ldots,x_m$. A Lie algebra $\fG$
{\it satisfies the polynomial identity $P$}
if 

$$P(a_1,\ldots,a_m)=0, \,\forall a_1,\ldots, a_m\in \fG,$$

\noindent see the monograph \cite{Bahturin}. The following Lemma is well-known. 

\begin{lemma} \label{noidentity}
 Let $P(x_1,\ldots,x_m)$ be a Lie
polynomial, and 
$(\fs_1, \fs_2,\ldots )$ be a countable family of finite dimensional simple Lie algebras such that

$$\lim_{i \to \infty}\dim_{L_i}\,\fs_i=\infty,$$

\noindent where $L_i:=\Cent(\fs_i)$.

Then $\fs_i$ does not satisfies $P$ for $i\gg0$.
\end{lemma}

\begin{proof} Indeed it is an easy consequence of
the celebrated paper of Amitsur and Levitzki
\cite{Amitsur-Levitzki}. For the sake of the reader, a quick proof is given below.

First we prove that
$\fgl_{N_0}$ does not satisfies 
$P$ for a certain integer $N_0>0$.
Let $F_m$ be the free Lie algebra over
$x_1,\cdots,x_m$. Since $F_m$ is residually finite dimensional, there is a finite codimensional ideal  $\fm\subset F_m$ such that
$P\notin\fm$. By Ado's theorem \cite{Bourbaki}, there is an embedding 
$F_m/\fm\subset \fgl_{N_0}$ for some integer $N_0$.
Thus $\fgl_{N_0}$ does not satisfies  $P$.

Next we observe that any polynomial identity 
of $\fs_i$ is also a polynomial identity for
the simple Lie algebra $\fs_i\otimes_{L_i} \overline{K}$. So we can add the simplifying hypothesis that $K$ is algebraically closed.
When $\dim \fs_i>248$, the Lie algebra
$\fs_i$ is isomorphic to 
$\fsl(n+1), \fso(2n+1), \fsp(2n)$ or $\fso(2n)$
for some $n$. Each of these Lie algebras contains $\fgl_{N_0}$ whenever $n\geq N_0$, therefore $\fs_i$ does not satisfies $P$ for $i\gg0$. 
\end{proof}

\begin{cor}\label{noPI} 
Any   residually semisimple weakly Noetherian
Lie algebra  $\fG$ does not satisfy any polynomial identity, except if $\fG$ is finite dimensional.
\end{cor} 

\begin{proof} Assume that $\fG$ is infinite dimensional.
By theorem \ref{structure1},
$\fG$ admits countably many simple quotients 
of finite dimension, say $\fs_1, \fs_2,\ldots$.
Moreover 

$$\lim_{i \to \infty}\dim_{L_i}\,\fs_i=\infty,$$

\noindent where $L_i:=\Cent(\fs_i)$. Thus by
Lemma \ref{noidentity}, $\fG$ does not satisfy any polynomial identity.
\end{proof}

\section{Characteristic ideals}\label{Charac}
Let $\fG$ be a weakly Noetherian Lie algebra.
In the present section we prove Theorem \ref{characteristic}, namely 
that any maximal ideal $\fm$ of 
codimension $>1$ is characteristic, which is
a crucial step in the proof of  Theorem A. 
When $\fm$ has infinite codimension, the result is a simple consequence of Lemma \ref{closure}(b). 

The case when $\fm$ has finite codimension is the
difficult part of the proof. It combines
the previous Theorem \ref{structure1} with some considerations about affine Lie algebras, defined below. The proof slighty differs from
the explanations of the Introduction. However, in Remark \ref{rk}, we briefly explains why the simplified version of Theorem \ref{structure1} is enough for the proof.

\subsection{Affine Lie algebras}

Let $\fs$ be a finite dimensional simple Lie algebra. A Lie algebra $\fG$ is called {\it affine} if

\begin{enumerate}
\item[(a)] $\fG$ contains $\fs$, and

\item[(b)] viewed as an $\fs$-module,  $\fG$ is a direct sum of adjoint modules.
\end{enumerate}

This kind of Lie algebras first appeared in 
Tits work \cite{Tits1}\cite{Tits2}. 
Over an algebraically closed field, or, more generally
when $\fs$ is a split simple Lie algebra,
affine Lie algebras are special
case of the Lie algebras graded by a root system
(LAGRS), which have been extensively investigated.
We can mention the early works of Koecher 
 \cite{Koe}, Kantor \cite{Kan} for rank one root system,  and, in the general case, there are a lot of papers, including  those of Berman and Moody \cite{BerMoo} or Benckart and Zelmanov \cite{BenZel}.  The main  difference  with the 
axioms of  \cite{Tits1}\cite{Tits2}  is the fact that affine Lie algebras do not contain trivial representations of $\fs$. Also, at the difference of other precited works, we do not assume that $\fs$ is split, so  there are no root decomposition in our setting.

Henceforth, it will be more easy to give direct proofs than extracting results from previous works. 

\bigskip
The first step  consists in the computation 
of the space
$\Hom_\fs(\fs\otimes\fs,\fs)$ 
 of all equivariant products on
$\fs$.  The Lie bracket
$[,]$  is a skew symmetric equivariant product.
When $\fs=\fsl_n$, there is a also a  
symmetric product $\cdot$ defined by

\centerline{$ x\cdot y:=1/2(xy+yx)-1/n\Tr(xy) 1_n$
for any $x, y\in \fsl_n$.}

\noindent Here $xy$ is the usual product of $n\times n$ matrices, $1_n$ denotes the identity of
$\Mat_n(K)$, and the additional term $1/n\Tr(xy) 1_n$ insures that
$x\cdot y$ is in $\fsl_n$. For $n\geq 3$, the product $\cdot$ is not zero.

The starting point is the following 
well-known observation:

\begin{lemma} \label{bilinear} Assume that $K=\overline K$ is algebraically closed.

If $\fs=\fsl_n$ with $n\geq 3$, then
$\Hom_\fs(\fs\otimes\fs,\fs)$ has dimension two,  with basis
the Lie bracket $[,]$ and the symmetric product $\cdot$.

Otherwise $\Hom_\fs(\fs\otimes\fs,\fs)$ is the one-dimensional space generated by the Lie bracket $[,]$.
\end{lemma}

\begin{proof} Let $\{e_i, f_i, h_i\mid i\in I\}$ be the standard Chevalley
generators of $\fs$, let 
$\fh=\oplus_{i\in I}\,K.h_i$ be the Cartan subalgebra and let
$\fb$ be   Borel subalgebra generated by
$\fh$ and the elements $e_i$. Let $\theta$ be the highest root
and let $e_\theta$ be the corresponding root element.

Let  $\pi:\fs\otimes\fs\to\fs$  be any nonzero equivariant product. Since the $\fs$-module $\fs\otimes \fs$ is generated by $Ke_\theta\otimes\fs$, we have

\centerline{$\pi(Ke_\theta\otimes\fs)\neq 0$.}

\noindent  Since
$Ke_\theta$ is the unique simple $\fb$-submodule of $\fg$,
the $\fb$-module $\pi(Ke_\theta\otimes\fs)$  intersects $Ke_\theta$. By weight considerations, we conclude that

\centerline{$\pi(Ke_\theta\otimes\fh)\neq 0$.}

Set $J:=\{i\in I\mid \exists \phi\in 
\Hom_\fs(\fs\otimes\fs,\fs) \,
\phi(e_\theta\otimes h_i)\neq 0 \,  \}$.
Since  $\pi(Ke_\theta\otimes\fh)$ lies in the one-dimensional space $K e_\theta$, we deduce that

\centerline{$\dim \Hom_\fs(\fs\otimes\fs,\fs)
\leq \Card \,J$.}

\noindent For $i\in I$ with $\theta(h_i)=0$, we have $[e_\theta,f_i]=0$. Obviously
 $\pi(e_\theta\otimes e_i)=0$, hence

\centerline{$0=[\pi(e_\theta\otimes e_i), f_i])=
\pi(e_\theta\otimes h_i)$.}

\noindent It follows that

\centerline{$J\subset \{i\in I\mid \theta(h_i)\neq 0\}$.}
  
If $\fs=\fsl_n$ with $n\geq 3$, the set 
$\{i\in I\mid \theta(h_i)\neq 0\}$ is a doubleton,
which proves the claim in that case. Otherwise,
the set 
$\{i\in I\mid \theta(h_i)\neq 0\}$ is a singleton,
which concludes the proof.
\end{proof}

We will now extend the result over 
nonalgebraically closed fields. Let $L$ be the centroid of the simple Lie algebra $\fs$. By Lemma \ref{com}, the division algebra $L$ is a field.

We can multiply any
$\mu\in \Hom_\fs(\fs\otimes\fs)$ by elements  $a\in L$, that is  $a\mu$ is the bilinear map

$$a\mu:x,y\in\fg\mapsto a\mu(x,y).$$

\noindent Therefore $\Hom_\fs(\fs\otimes\fs)$ is a $L$-vector space. A priori, we can define
another structure of $L$-vector space on
$\Hom_\fs(\fs\otimes\fs)$ by using the action of $L$ on the first argument, but Assertion (a)
of the next Lemma shows that it defines the same structure of $L$-vector space.

\begin{lemma}\label{bilinear2} Let $\fs$ be a simple finite dimensional Lie algebra and let $L$ be its centroid.

\begin{enumerate}
\item[(a)] Any $\mu\in \Hom_\fs(\fs\otimes\fs,\fs)$ is $L$-bilinear,
\item[(b)] if 
$\overline{K}\otimes_L \fs$ is isomorphic to 
$\fsl_n(\overline{K})$ for some $n\geq 3$,
then the $L$-vector space 
$\Hom_\fs(\fs\otimes\fs,\fs)$ has dimension two,
with basis the Lie bracket $[-,-]$ and a 
certain symmetric product $\cdot$,
\item[(c)] otherwise, the $L$-vector space 
$\Hom_\fs(\fs\otimes\fs,\fs)$ is generated by the Lie bracket $[-,-]$.

\end{enumerate}
\end{lemma}

\begin{proof} By definition of the centroid, the bracket is $L$-bilinear. Therefore Assertion (a) is proved when $\Hom_\fs(\fs\otimes\fs,\fs)$ 
has dimension one over $L$.

If $\overline{K}\otimes_L \fs$ is not isomorphic to $\fsl_n$ for some $n\geq 3$, then 
by Lemma \ref{bilinear},  the space
$\Hom_\fs(\fs\otimes\fs,\fs)$ has $K$-dimension
$[L:K]$, so it is generated by the bracket as
vector space over $L$, which proves Assertion (c).

Otherwise, by Lemma \ref{bilinear},  the space
$\Hom_\fs(\fs\otimes\fs,\fs)$ has $K$-dimension
$2[L:K]$ and it contains a symmetric product.
 Thus
$\Hom_\fs(\fs\otimes\fs,\fs)$ is a $L$-vector space generated by the  Lie bracket and some symmetric product $\cdot$, which proves assertion (b).

We also observe that the Lie algebra 
$\overline{K}\otimes \fs$ is a direct sum of
$[L:K]$ copies of $\overline{K}\otimes_L \fs$.
It is clear that the symmetric product on
$(\overline{K}\otimes_L \fs)^{[L:K]}$ is
bilinear with respect to 
${\overline K}\otimes L\simeq  {\overline K}
^{[L:K]}$. It follows that any equivariant symmetric
product in $\Hom_\fs(\fs\otimes\fs,\fs)$
is $L$-bilinear, which completes the proof of Assertion (a).
\end{proof}

Let $\fs$ be a simple finite dimensional Lie algebra, let $\fG\supset\fs$ be an affine Lie algebra over $\fs$ and set

$$A:=\Hom_\fs(\fs,\fG).$$ 

We observe that $A$ is a $L$-vector space,
where $L=\Cent(\fs)$, and there is a natural isomorphism of $\fs$-modules

$$\fG\simeq \fs\otimes_L A.$$

\noindent For $x\in\fs$ and $a\in A$, we denote the tensor $x\otimes a$ as $x(a)$. We also denote by $1$ the given embedding $\fs\to\fg$,
therefore $x(1)=x$, for any $x\in \fs$.

\begin{lemma}\label{aff} With the previous notations,
$A$ carries a structure of  unital $L$-algebra
with unit $1$. Moreover,  
 the bracket is expressed as follows:
\begin{enumerate}
\item[(a)] if $\fs\otimes_L \overline{K}
\simeq \fsl_n$ for some $n\geq 3$, then
$$[x(a),y(b)]=1/2\big([x,y](a\circ b+b\circ a)
+(x\cdot y)(a\circ b-b\circ a)\big)$$

\item[(b)] otherwise, $A$ is commutative and
$$[x(a),y(b)]=[x,y](a\circ b)$$
\end{enumerate}

\noindent where $\circ$ denotes the product on $A$.
\end{lemma}

\begin{proof} The bracket of $\fG$ is an
element of

$$\Hom_\fs(\fs\otimes\fs,\fs)
\otimes_L \Hom_L(A\otimes A, A).$$  

\noindent Assume now that  
$\fs\otimes_L \overline{K}
\simeq \fsl_n$ for some $n\geq 3$. Then 
by Lemma \ref{bilinear2},
$A$ is endowed with two products $\circ_1$ and
$\circ_2$ such that

$$[x(a),y(b)]=[x,y](a\circ_1 b)+
(x\cdot y)(a\circ_2 b),$$

\noindent for any $x,y\in\fg$ and $a,b\in A$.
Since the Lie bracket is skew-symmetric,
the product $\circ_1$ is commutative and 
$\circ_2$ is skew-symmetric. If we set

$$a\circ b= a\circ_1 b+a\circ_2 b,$$

\noindent we deduce the required formula for
$[x(a),y(b)]$.

When $\fs\otimes_L \overline{K}
\not\simeq \fsl_n$ for some $n\neq 2$,
it follows from Lemma \ref{bilinear2} that

 $$[x(a),y(b)]=[x,y](a\circ b),$$

\noindent for some commutative product $\circ$ on $A$.

On both cases, it is clear that $1$ is a unit.
\end{proof}

\noindent {\it Remark:} The previous proof only uses the skew symmetry of the bracket. Therefore the brackets defined by the previous lemma not need satisfy Jacobi identity. 

\begin{lemma}\label{corres} Let 
$\fG=\fs\otimes_L A$ be an affine Lie algebra as before. Then any ideal $\fm$ of $\fG$ is of the form

$$\fm=\fs\otimes_{L} {\bf m},$$

\noindent for some two-sided ideal $\bf{m}$ of $A$.
\end{lemma}

\begin{proof} Let $\fm$ be an ideal of $\fG$.
Since $\fm$ is $\fs$-invariant, we have
$\fm=\fs\otimes_L {\bf m}$ for some
$L$-vector subspace ${\bf m}\subset A$.
We will show that ${\bf m}\otimes_L\overline{K}$
is an ideal of $A\otimes_L\overline{K}$, therefore we can assume that $K=\overline{K}$
is algebraically closed.

First assume that $A$ is commutative. Pick
any $x, y\in \fs$ with $[x,y]\neq 0$. Then the formula 
$[x(a),y(b)]=[x,y](a\circ b)$ obviously implies
that ${\bf m}$ is an ideal.

Otherwise, $\fs\simeq \fsl_n$ for some 
$n\geq 3$. We decompose the product
$a\circ b$ of two elements of $A$ as

$$a\circ b=a\circ_1 b+ a\circ_2 b,$$

\noindent where $\circ_1$ is the symmetric part  and $\circ_2$ the antisymmetric part of the product $\circ$. We will prove that
${\bf m}$ is an ideal for each of the products
$\circ_1$ and $\circ_2$.
It does not matter the way   the symmetric product $x.y$ is normalized but we can assume that $x.y=1/2(xy+yx) -1/n \Tr(xy) 1_n$.

First choose a diagonal matrix $h\in\fsl_n$
such that $h^2$ is not proportional to $1_n$,
e.g. $h=\diag(1,-1,0,\ldots)$.
In follows that  $h.h\neq 0$. Since

$$[h(a),h(b)]=h.h(a\circ_2 b),$$

\noindent we deduce that $\bf{m}$ is an ideal for the product $\circ_2$. 

Next let $I,J\in \fsl_n$ be Paoli matrices.
They satisfy $IJ=-IJ\neq 0$. Since $\Tr(IJ)=0$, it follows  that 
$I.J=0$. From the relation

$$[I(a),J(b)]=2IJ(a\circ_1 b),$$

\noindent we deduce that $\bf{m}$ is an ideal for the product $\circ_1$, which completes the proof.
\end{proof}

\begin{lemma}\label{der} Let $\fs$ be a simple Lie algebra 
with centroid $L$, and let 
$\fG=\fs\otimes_L A$ be an affine Lie algebra.

Any derivation $\partial$ of $\fG$ is of the form

$$x(a)\mapsto [X,x(a)] + x(\delta a),$$

\noindent where $X$ belongs to $\fG$ and
$\delta$ is a derivation of $A$.
\end{lemma}

\begin{proof} 

The map $\partial\vert_{\fs}$
$$x\in \fs\mapsto \partial x\in\fG$$

\noindent is a derivation of $\fs$. By Cartan's Theorem \ref{cartan}(e), we have

$$H^1(\fs,\fG)=H^1(\fs,\fs)\otimes_L A=
\Out(\fs)\otimes_L A=0.$$

\noindent Hence, there is $X\in\fG$ such that

$$\partial x=[X,x], \, \forall x\in\fs.$$

\noindent Set $\partial'=\partial-\ad(X)$. 
Since $\partial'\fs=0$, the linear map 
$\partial'$ is $\fs$-equivariant. Thus
$\partial'$ is $L$-linear and 

$$\partial' x(a)=x(\delta a)$$ 

\noindent for some 
$L$-linear map $\delta:A\to A$.

We now show that  $\delta$ is a derivation of
$A$. Indeed 
we will only consider the case where  
$\overline{K}\otimes_L\fs\simeq\fsl_n$ for some $n\geq 3$, otherwise the proof is straightforward. First decompose the product $\circ$ of $A$ as $\circ=\circ_1+\circ_2$, where $\circ_1$ is commutative and $\circ_2$ is skew-symmetric. We will show that $\delta$ is a derivation of both products $\circ_1$ and $\circ_2$.

Let $x,y\in \fs$ with $[x,y]\neq 0$. Using
that 

$$[x(a),y(b)]+[x(b),y(a)]=2[x,y](a\circ_1 b),$$

\noindent we easily deduce that 

$$\delta(a\circ_1 b)=
\delta a\circ_1 b+a\circ_1 \delta b.$$

Similarly, consider an element $h\in \fs$ such that $h\cdot h\neq 0$. Then using that

$$[h(a),h(b)]=(h\cdot h)(a\circ_2 b),$$

\noindent we easily deduce that 

$$\delta(a\circ_2 b)=
\delta a\circ_2 b+a\circ_2 \delta b.$$

\noindent It follows that $\delta$ is a derivation, which completes the proof.
\end{proof}

For a given derivation $\partial$ of
$\fs\otimes_L A$,  $\delta$ 
is called {\it the derivation of $A$ induced by $\partial$}.

\subsection{Complete algebras with a derivation}\label{complete}

Let $A$ be a nonassociative  algebra and let

\centerline{${\bf m}_*: {\bf m}_1\supset {\bf m}_2\cdots$}

\noindent  be a decreasing 
sequence of two-sided ideals. 
The algebra $A$ is called {\it ${\bf m}_*$-complete} if 

\centerline{$A \simeq \varprojlim A/{\bf m}_n$.}

\noindent A typical example is the algebra of formal series  $K[[t]]$ in the variable $t$,  where ${\bf m}_*$ be the
decreasing  sequence of ideals $(t)\supset (t^2)\supset\cdots$. 

Let $A$ be a  nonassociative algebra, let
${\bf m}_1$ be a two-sided ideal and 
let $\partial:A\to A$ be a derivation. Let 

\centerline{${\bf m}_1\supset {\bf m}_2\cdots$}

\noindent  be the decreasing 
sequence of subspaces inductively defined by

\centerline{
${\bf m}_{k+1}:=\{x\in {\bf m}_k\mid \partial x\in{\bf m}_k\}$.}

\noindent As ${\bf m_1}$ is a two-sided ideal, it is proved by induction  that all subspaces ${\bf m}_k$ are two-sided ideals satisfying

\begin{align}
{\bf m}_k.{\bf m}_l&\subset {\bf m}_{k+l}\,\,
\forall \,k,l>0. \label{ideal-condition}
\end{align}

\begin{lemma}\label{formalseries} Assume that 

\begin{itemize}
\item[(a)] The algebra $A$ is unital and 
$A=K.1\oplus {\bf m}_1$

\item[(b)] $\partial {\bf m}_1\not\subset {\bf m}_1$

\item[(c)] $A$ is ${\bf m}_*$-complete.
\end{itemize}

Then $A$ is commutative and associative and we have
$$A\simeq K[[t]].$$
\end{lemma}

\begin{proof} Since 
$\cap_{n\geq 0}\,{\bf m}_n=0$, we conclude
that ${\bf m}_1$ contains no 
nonzero $\partial$-invariant subspace.

Let $\Com(A)$ be the linear span of commutators
$\{ ab-ba\mid a,b \in A\}$. Any $a,b\in A$ can be writen as $a=x.1+\overline{a}$, 
$b=y.1+\overline{b}$ where $\overline{a}$ and  
$\overline{b}$ belong to ${\bf m}_1$ and where
$x, y$ are scalars. Since

$$ab-ba=\overline{a} \overline{b}-
\overline{b}\overline{a}.$$

\noindent we deduce that $\Com(A)$ lies in 
${\bf m}_2$. Since it is invariant by 
$\partial$, we conclude that 
$\Com(A)=0$, that is $A$ is commutative.

Let $\Ass(A)$ be the linear span of associators $\{ (ab)c-a(bc)\mid a,b,c \in A\}$.
We similarly prove that 
$\Ass(A)$ is $\partial$-invariant subspace of ${\bf m}_3$. Therefore $\Ass(A)=0$, that is $A$ is associative.

It follows that $A$ is an homorphic image of
$K[[t]]$ and that ${\bf m}_1$ is the ideal $(t)$. Since the ideal
$(t)$ in $K[[t]]/(t^n)$  is invariant by any derivation, we conclude that $A=K[[t]]$.
\end{proof}

\begin{lemma}\label{lift} 
Let $\fG$ be a Lie algebra and let ${\fm}_*: {\fm}_1\supset {\fm}_2\cdots$ be a decreasing 
sequence of ideals satisfying

$$[\fm_k,\fm_l]\subset \fm_{k+l}\,\,\forall \,k,l>0. $$

 Assume 

\begin{itemize}
\item[(a)] The Lie algebra $\fG$ is ${\fm}_*$-complete,
\item[(b)] the Lie algebra $\fs:=\fG/{\fm}_1$ is simple of finite dimension, and
\item[(c)] $\dim {\fm}_k/{\fm}_{k+1}<\infty$ for all $k\geq 1$.
\end{itemize}

Then $\fG\simeq \fs\ltimes{\fm}_1$.
\end{lemma}

\pf Let $k\geq 1$. By hypotheses, the Lie algebra 
$\fG/{\fm_k}$ is finite dimensional and $\fm_1/\fm_k$ 
is a nilpotent ideal. By
Theorem \ref{cartan}(b), we have

\centerline{
$\fg/{\fm_k}\simeq \fs\ltimes {\fm_1}/{\fm_k}$.}

\noindent Moreover any lift of $\fs$ modulo ${\fm_k}$
can be extended to a lift modulo ${\fm}_{k+1}$, from which the claim follows.
\epf

Let $\fG$ be a Lie  algebra, let $\fm_1$ be an ideal and
let $\partial:\fG\to \fG$ be a derivation. 
As before, we define the  decreasing 
sequence of ideals

\centerline{${\fm}_1\supset {\fm}_2\cdots$}

\noindent  by induction as

\centerline{${\fm}_{k+1}:=\{x\in {\bf m}_k\mid \partial x\in{\fm}_k\}$.}

\begin{lemma}\label{aff[[t]]} Assume 
\begin{enumerate}
\item[(a)] $\fs:=\fG/\fm_1$ is a finite dimensional simple Lie algebra,
\item[(b)] $\partial \fm_1\not\subset\fm_1$
\item[(c)] $\fG$ is $\fm_*$-complete.
\end{enumerate}

Then $\fG\simeq \fs\otimes K[[t]]$.
\end{lemma}

\begin{proof}
By Lemma \ref{lift}, we have

$$\fG\simeq \fs\ltimes \fm_1.$$

The derivation $\partial$ induces injective morphisms of $\fs$-modules

\begin{align*}
\partial:\fm_{1}/\fm_{2}&\to
\fG/\fm_{1}=\fs\\
\partial:\fm_{2}/\fm_{3}&\to
\fm_{1}/\fm_{2}\\
&\cdots\\
\end{align*}

\noindent therefore each $\fs$-module 
$\fm_{k}/\fm_{k+1}$ is isomorphic to $\fs$ or to
$\{0\}$. Since any finite dimensional $\fs$-module is
semisimple, we deduce that $\fm_k$ is a direct sum of adjoint modules. 

It follows that  $\fG$ is a direct sum of adjoint modules, that is $\fG$ is an affine algebra.
Let $L$ be the centroid of $\fs$. By Lemma \ref{aff}, we have

$$\fg=\fs\otimes_L A,$$

\noindent for some unital $L$-algebra $A$.
By Lemma \ref{corres}, there is a decreasing sequence of  ideals ${\bf m}_k$ of $A$
such that 

$$\fm_k=\fs\otimes_L {\bf m}_k\,\forall k>0.$$

By Lemma \ref{der}, $\partial$ iduces a derivation $\delta$ of $A$. It is clear that
the triple $(A, \delta, m_*)$ satisfies the hypotheses of Lemma \ref{formalseries}.
Hence we have $A=L[[t]]$, and

$$\fg\simeq \fs\otimes_L L[[t]]= \fs\otimes K[[t]].$$
\end{proof}

\subsection{Maximal ideals of finite codimension}

In this subsection, we  finally  prove that, given a weakly
Noetherian Lie algebra $\fG$,
 any maximal ideal $\fm$ with
 $1<\codim\, \fm<\infty $ is characteristic.
 We start with some preparatory results.

\begin{lemma}\label{cartancor2} Let $\fs$ be a simple  finite dimensional Lie algebra with centroid $L$.
Let $\fa\subset\fs$ be a Lie subalgebra which contains a $L$-basis of $\fs$.

Then  

\begin{enumerate}
\item[(a)] $\fa$ is a simple Lie algebra,
\item[(b)] its centroid $L':=\Cent(\fa)$ satisfies
$K\subset L'\subset L$, and
\item[(c)] $\fa \otimes_{L'}L=\fs$.
\end{enumerate}
\end{lemma}

\begin{proof} Let $\fr$ be a minimal ideal
of $\fa$, and let $L.\fr$ be the $L$-vector space generated by $\fr$. Since $\fa$ contains an $L$-basis of $\fs$, the space $L.\fr$ is an
ideal of $\fs$, thus $L.\fr=\fs$. 

We claim that $\fr$ is a simple Lie algebra.  
Since $L.\fr=\fs$, the ideal $\fr$ cannot be abelian.
By minimality of $\fr$, we conclude that $\fr$ is simple.

By Cartan Theorem \ref{cartan}(e), we have $\Out(\fr)=0$.
Therefore by Lemma \ref{centralizer}, the Lie algebra
$\fa$ is a direct sum

$$\fa\simeq C_\fa(\fr)\oplus \fr.$$

\noindent Since $\fr$ contains a basis of $\fs$, its centralizer
$C_\fa(\fr)$ is trivial, from which we deduce that
$\fa=\fr$. Therefore we have proved that $\fa$ is simple.

It is clear that 

$$L'=\End_\fa(\fa,\fa)\subset \End_\fa(\fa,\fs)=L,$$

\noindent therefore $L'$ is a subfield of $L$. Moreover,
by Corollary \ref{simplecentroid}(b), $\fa\otimes_{L'} L$ is a simple Lie algebra, therefore the homomorphism
$\fa\otimes_{L'} L\to \fs$ is an isomophism, which completes the proof.
\end{proof}

Given   a simple finite dimensional Lie algebra $\fs$, we view
$$\fs\otimes K[[t]]=\varprojlim\, \fs\otimes K[t]/(t^n)$$ 

\noindent as a  topological Lie  algebra.

\begin{lemma}\label{reduction} Any dense Lie subalgebra $\fG\subset \fs\otimes K[[t]]$ contains a finitely
generated dense Lie subalgebra.
\end{lemma}

\pf In fact  a Lie subalgebra $\fG\subset \fs\otimes K[[t]]$ 
is dense whenever the induced homomorphism

\centerline{$\fG\to \fs\otimes K[[t]]/(t^2)$}

\noindent is surjective, from which the claim follows.
\epf

\begin{cor}\label{dense2} No dense subalgebra $\fG\subset \fs\otimes K[[t]]$ is  weakly
Noetherian.
\end{cor}

\pf  Let $\fG$ be a dense subalgebra of

$$\fs\otimes K[[t]]\simeq \fs\otimes_L L[[t]],$$

\noindent where  $L=\Cent(\fs)$.
By Lemma \ref{reduction}, we can assume that 
 $\fG$ is finitely generated. Therefore 

\centerline{$\fG\subset \fs\otimes_L A$,}

\noindent for some finitely generated unital $L$-subalgebra $A$ of $L[[t]]$. 

Let $X$ be the set of closed points in 
$\Spec\,A$. For each $x\in X$, let
${\bf m}_x$ be the corresponding maximal ideal of $A$ and let $L_x:=A/{\bf m}_x$ be the residue field at $x$. We also denote as $0$ the point of $X$ defined by the maximal ideal
${\bf m}_0:=A\cap t L[[t]]$.

For $x\in X$ and $z\in\fG$, set 

\begin{align*}
{\fm}_x:=&\fG\cap \fs\otimes_L {\bf m}_x\\
{\fG}_x:=&\fG/{\fm}_x\\
K_x:=&\Cent({\fG}_x)
\end{align*}

\noindent and let $z(x)\in\fG_x$ be the value of $z$ at $x$, 
that is $z(x):=z\mod {\fm}_x$.

\smallskip
 First, we claim that there is an open dense subset 
$X^0\subset X$ such that
${\fG}_x$ is a simple Lie algebra of dimension 
$d$ over $K_x$.

Set $d:=\dim_L\fs$ and let $y_1,\ldots, y_d$ be an $L$-basis of $\fs$. Since $\fG$ is dense, we have $\fG_0=\fs$. Therefore there are  $z_1,\ldots z_d\in\fG $
such that 

$$z_1(0)=y_1,\ldots, z_d(0)=y_d.$$

\noindent Since $\fs\otimes_L A$ is a free $A$-module 
with basis $y_1,\ldots, y_d$, we have

$$z_1\wedge_A\ldots\wedge_A z_d
=f\, y_1\wedge_A \ldots\wedge_A y_d$$

\noindent for some $f\in A$. Set

$$X^0=\{x\in X\mid f(x)\neq 0\}.$$

\noindent Since $A$ has no zero-divisors, $X$ is irreducible. Since $0$ belongs to $X^0$,
the open set $X^0$ is dense.

By definition, we have $\fG_x\subset \fs\otimes_L\, L_x$
and, for $x\in X^0$,  the elements
$z_1(x),\ldots, z_d(x)\in \fG_x$ form  a
$L_x$ basis of the simple Lie algebra $\fs\otimes_L L_x$.
\noindent Therefore, by Lemma \ref{cartancor2}, 
$\fG_x$ is a simple Lie algebra of dimension $d$ over its centroid, for all $x\in X^0$, which proves the claim.

\smallskip 
Since $A$ is finitely generated, by Hilbert 
Nullstellensatz \cite{Hilbert0}, we have

$$\cap_{x\in X^0}\,{\bf m}_x=0.$$ 

\noindent Therefore

$$\cap_{x\in X^0}\,{\fm}_x=0.$$

\noindent Since $\fG$
is a dense subalgebra of $\fs\otimes K[[t]]$,
its dimension is infinite. The ideals ${\fm}_x$ need not be  pairwise distinct, but we can conclude than the set

$$\{{\fm}_x\mid x\in X^0\}$$

\noindent is infinite. Henceforth $\fG$ has infinitely many distinct simple quotients $\G_x$ of dimension $d$
over its centroid.  By Theorem
\ref{structure1}, the Lie algebra $\fG$ is not
weakly Noetherian.
\end{proof}

\begin{remark}\label{rk} The previous proof uses the refined version
of Theorem \ref{structure1}. 

In fact, a refined version of Hilbert's Nullstellensatz implies
that there is an integer $e$ such that the closed points of degree
$e$ are dense in $X^0$.  Hence the simplified version of
Theorem \ref{structure1} stated in the introduction is enough.
 \end{remark}

\begin{prop}\label{propfinitecodim} Let $\fG$ be a weakly Noetherian
Lie algebra. Any  maximal ideal $\fm$
with
$$1<\codim\, \fm<\infty,$$
\noindent is characteristic.
\end{prop}

\begin{proof}
Assume otherwise. Thus there is a derivation 
$\partial$ of $\fG$ such that 

$$\partial \fm\not\subset\fm.$$

Set $\fm_1=\fm$  and let

\centerline{${\fm}_1\supset {\fm}_2\cdots$}

\noindent  be the decreasing 
sequence of subspaces defined by induction as

\centerline{
${\fm}_{k+1}:=\{x\in {\bf m}_k\mid \partial x\in{\fm}_k\}$.}

\noindent As ${\fm_1}$ is an ideal, it has been observed 
at the begining of  Subsection \ref{complete} that all ${\fm}_{k}$ are ideals and

$$[\fm_k,\fm_l]\subset \fm_{k+l},\,\,\forall k,l>0.$$

Set

$$\overline{\fG}=\varprojlim\,\fG/\fm_k,$$

\noindent and, for $k\geq 1$,  let $\overline{\fm}_k$ be the closure of $\fm_k$ in $\overline{\fG}$. The derivation $\partial$ extends to $\overline{\fG}$ and we have

\begin{align*}
\overline{\fm}_{k+1}&=\{x\in 
\overline{\fm}_k\mid 
\partial x\in\overline{\fm}_k\}\\
[\overline{\fm_k},\overline{\fm_l}]&\subset 
\overline{\fm_{k+l}},
\end{align*}

\noindent for all $k, l>0$.
Therefore by Lemma \ref{aff[[t]]},
we have

$$\overline{\fG}\simeq \fs\otimes K[[t]],$$

\noindent where $\fs=\fG/\fm$.

Without loss of generality, we can assume that
$\cap_{k>0}\,\fm_k=0$, that is $\fG\subset \overline{\G}$ Since $\fG$ is a dense 
 subalgebra of $\fs\otimes K[[t]]$, the Lie algebra $\fG$ is not weakly Noetherian by  Corollary \ref{dense2}.
\end{proof}

\subsection{Maximal ideals of arbitrary codimension}

We can now prove the main result of the Section:

\begin{theorem}\label{characteristic} Let $\fG$ be a weakly Noetherian Lie algebra. Then any
maximal ideal  of codimension $>1$ is characteristic.
\end{theorem}

\begin{proof} Let $\fm$ be maximal ideal  of codimension $>1$. If $\fm$ has finite codimension, then 
$\fm$ is characteristic by Proposition \ref{propfinitecodim}.

Assume now that  $\fm$ has infinite codimension.
By  Lemma \ref{closure}(a), we have

$$\fm=\Cl_{\fG}(\fm),$$

 \noindent and, by Lemma \ref{closure}(b),
 $\fm$ is a characteristic ideal.
 \end{proof}

\section{Proof of Theorem A}\label{proofA}

In the section, we now prove a refined version  Theorem A,
which holds for an arbitrary weakly Noetherian Lie algebra $\fG$.

Indeed,  we define  a filtration

$$\fG\supset \fG_{(0)}\supset \fG_{(1)}\supset \ldots\fG_{(\alpha)}\supset\ldots$$ 

\noindent of $\fG$ as follows:

\begin{enumerate}
\item[(a)] \hskip27mm$\fG_{(0)}=\cD^*\rad^f(\fG),$
\item[(b)] if the ordinal $\alpha$ is a successor,
that is $\alpha=\beta+1$ for some $\beta$, then
$$\fG_{(\alpha)}=
[\rad(\fG_{(\beta)}),\rad(\fG_{(\beta)})]$$
\item[(c)] If $\alpha\neq 0$ is a limit ordinal, then
$$\fG_{(\alpha)}=[\cap_{\beta<\alpha}\,\rad(\fG_{(\beta)}), \cap_{\beta<\alpha}\,\rad(\fG_{(\beta)})].$$
\end{enumerate}

\noindent When $\fG$ satisfies the simplifying hypothesis of the 
introduction, the definition coincides with the previous definition.

\subsection{Outer derivations}
We start with an elementary fact.

\begin{cor}\label{quot} Let $\fG$ be a Lie algebra,
let $\fm$ be an ideal of $\fG$ and let
$\fr$ be a characteristic ideal of $\fm$.
 Assume that the Lie algebra 
$\fs:=\fm/\fr$ is a simple.
Then

\begin{enumerate}
\item[(a)] If $\fm\subset\rad\fG$, we have
$\Out(\fs)\neq 0$,
\item[(b)] If $\fm\subset\rad^f\fG$, the Lie algebra
$\fs$ is infinite dimensional.
\end{enumerate}
\end{cor}

\begin{proof} The hypothesis implies that
$\fr$ is indeed an ideal of $\fG$.
Without loss of generality, we can assume that $\fr=0$.
So the ideal $\fm=\fs$ is a simple Lie algebra.

We prove Assertion (a) by contradiction, thus
assume that $\Out(\fs)=0$. By Lemma \ref{centralizer}(a), we have

$$\fG= C_\fG(\fs)\oplus \fs,$$

\noindent which implies that 
$C_\fG(\fs)$ is a maximal ideal of codimension $>1$. Therefore
$$C_\fG(\fs)\supset\rad\fG\supset \fs,$$
\noindent which is contradictory.

Similarily, we prove Assertion (b) by contradiction. So assume that $\fs$ is a finite dimensional simple Lie algebra.
We have $\Out(\fs)=0$ by Theorem \ref{cartan}(e), thus 
Lemma \ref{centralizer}(a) implies

$$\fG=C_\fG(\fs)\oplus \fs,$$

\noindent which means that 
$C_\fG(\fs)$ is a maximal ideal of finite codimension $> 1$. Therefore
$$C_\fG(\fs)\supset\rad^f\fG\supset \fs,$$
\noindent which is contradictory.
\end{proof}

\begin{lemma}\label{prepa} Let $\fG$ be a  Noetherian Lie algebra. Then for all ordinals $\alpha$,
\begin{enumerate}
\item[(a)] the Lie algebra $\fG_{(\alpha)}$ is a perfect ideal,
\item[(b)] Any simple quotient of $\fG_{(\alpha)}$ is infinite dimensional
\item[(c)] For $\alpha\geq 1$, no simple quotient of $\fG_{(\alpha)}$ is isomorphic to a Krichever-Novikov algebra.
\end{enumerate}
\end{lemma}

\begin{proof} Without loss of generality, we can assume that
$\fG_{(0)}\neq 0$.

\bigskip
\noindent {\it First step: proof of the lemma for $\alpha=0$.}
By definition 
$\fG_{(0)}:=\cD^*\rad^f(\fG)$ is
a perfect ideal. It remains to prove
that any simple quotient $\fs=\fG_{(0)}/\fr$ of $\fG_{(0)}$ is infinite dimensional. 
 By Theorem \ref{characteristic}, $\fr$ is a characteristic ideal of $\fG_{(0)}$ and, by definition, we have  $\fG_{(0)}\subset\rad^f(\fG)$.
Applying  Lemma \ref{quot}(b) to the Lie algebra $\fG$, its ideal
$\fm=\fG_{(0)}$ and its subideal $\fr$, we obtain that 
$\fs$ is infinite dimensional.

\bigskip
\noindent{\it Second step: proof, by transfinite
induction on $\alpha$, that $\fG_{(\alpha)}$ is
a perfect ideal.}  
Since $\fG_{(0)}$ is perfect and all simple quotients are infinite dimensional, we conclude that $\fG_{(0)}$  acts trivially on
any finite dimensional $\fG_{(0)}$-module $M$.

Since the case $\alpha=0$ is done, we now prove by induction that $\fG_{(\alpha)}$ is
a perfect ideal for any ordinal  $\alpha>0$.

First assume that $\alpha$ is a successor, that is $\alpha=\beta+1$ for some ordinal $\beta$.
By induction hypothesis, $\fG_{(\beta)}$ is an ideal.
By Theorem \ref{characteristic}, 
$\rad(\fG_{(\beta)})$ is a characteristic ideal
of $\fG_{(\beta)}$. Therefore 
$\rad(\fG_{(\beta)})$ is an ideal of $\fG$.

We now consider  $M:=\rad(\fG_{(\beta)})/\cD^2 \rad(\fG_{(\beta)})$ as an $\fG_{(0)}$-module.
 By Lemma \ref{basic}(e),  $M$ is finite dimensional,
 thus  $\fG_{(0)}$  acts trivially on $M$.
In particular we have 

$$\rad(\fG_{(\beta)}).M=0,$$ 

which amounts to
 
 $$\cD^2 \rad(\fG_{(\beta)})=\cD^1 \rad(\fG_{(\beta)}).$$
 
 \noindent Therefore 
 $\fG_{(\alpha)}=[\rad(\fG_{(\beta)}),\rad(\fG_{(\beta)})]$ is 
 a perfect ideal.
 
 Next assume that $\alpha$ is a limite ordinal.
 Set 
 
 $$\fr=\cap_{\beta<\alpha}\,\rad(\fG_{(\beta)})=\cap_{\beta<\alpha}\,\fG_{(\beta)}.$$
 
 \noindent By induction hypothesis, $\fr$ is an ideal
 of $\fG$.  As before, we can consider $M:=\fr/\cD^2\fr$ as an 
 $\fG_{(0)}$-module, which is finite dimensional by Lemma \ref{basic}(e). Since
 $\fG_{(0)}$ acts trivially on $M$, we conclude  that 
 $\fr.M=0$. Equivalently
 
 $$\cD^2 \fr=\cD^1 \fr.$$
 \noindent Therefore 
 $\fG_{(\alpha)}=[\fr,\fr]$ is perfect, which completes the proof that all Lie algebras $\fG_{(\alpha)}$ are perfect ideals.
 
 \bigskip
\noindent {\it Third step: proof that, for any $\alpha\geq 1$,
any simple quotient $\fs=\fG_{(\alpha)}/\fr$ of
$\fG_{(\alpha)}$
is infinite dimensional and is not a Krichever-Novikov algebra.}
 By Theorem \ref{characteristic}, $\fr$ is a characteristic ideal of $\fG_{(\alpha)}$ and, by definition, we have
 $\fG_{(\alpha)}\subset\rad(\fG_{(0)})$.
 Applying Lemma  \ref{quot}(a) to the Lie algebra 
 $\fG_{(0)}$, its ideal $\fm=\fG_{(\alpha)}$ and its subideal $\fr$,
 we conclude  that 
$$\Out(\fs)\neq 0.$$ 
 
 \noindent Therefore, $\fs$ is neither finite dimensional by Cartan's theorem \ref{cartan}(e) nor isomorphic to
 a Krichever-Novikov algebra by Grabowski's theorem 
 \ref{Grabowski}(b).
  \end{proof}
 
 \subsection{Proof of Theorem A}

 We now state a refined version of Theorem A.
 
\begin{mainA} Let $\fG$ be a weakly Noetherian Lie algebra. Then,
for any ordinal $\alpha$:

\begin{enumerate}
\item[(a)] the Lie algebra $\fG_{(\alpha)}$
is a  perfect ideal.

\item[(b)] 
$\fG_{(\alpha)}/\fG_{(\alpha+1)}$
is a central extension of 
$\fG_{(\alpha)}/\rad\fG_{(\alpha)}$ by a finite
dimensional center.

\item[(c)] For any integer $d$, $\fG/\fG_{(0)}$
admits only finitely many simple quotients $\fs$
with $\dim_{\Cent(\fs)}\, \fs=d$. 

In particular,  if $\dim \fG/\rad^f(\fG)=\infty$, the Lie algebra
$\fG/\rad^f(\fG)$ does
not satisfy any polynomial identity.

\item[(d)] Any
simple quotient  of $\fG_{(\alpha)}$ is infinite dimensional.

\item[(e)]
For $\alpha\geq 1$, no Krichever-Novikov algebras occur as a quotient of $\fG_{(\alpha)}$.
\end{enumerate}
\smallskip
Moreover if $\fG$ is Noetherian,
then $\fG_{(\alpha)}=0$ for $\alpha$ big enough.
\end{mainA}
 
 \begin{proof} Assertions (a)(d)(e) follow from
 Lemma \ref{prepa}. Assertion (c) follows from
 Theorem \ref{structure1} and   Corollary \ref{noPI}.
 
We now prove Assertion (b).
 By definition, $M:= 
 \rad\fG_{(\alpha)}/\fG_{(\alpha+1)}$ is an abelian section, thus $M$ is finite dimensional.
 Since $\fG_{(\alpha)}$ is perfect and does not admit finite dimensional simple quotient, 
we deduce that  $\fG_{(\alpha)}$ does not admit any nonzero finite dimensional quotient.
Therefore $\fG_{(\alpha)}$ acts trivially on $M$,
what amonts to the fact that 
$\fG_{(\alpha)}/\fG_{(\alpha+1)}$
is a central extension of 
$\fG_{(\alpha)}/\rad\fG_{(\alpha)}$ by a finite
dimensional center.

Next, assume that $\fG$ is Noetherian. 
By dimension considerations, it is clear that
the descending sequence $\fG_{(\alpha)}$ stabilizes, that is $\fG_{(\alpha+1)}=\fG_{(\alpha)}$ for some ordinal $\alpha$.

We claim that $\fG_{(\alpha)}=0$.
Otherwise, by Lemma \ref{idealACC}, $\fG_{(\alpha)}$
admits a maximal ideal $\fm$.
We have already proved that $\fG_{(\alpha)}$
is perfect,  so $\fm\supset \rad(\fG_{(\alpha)}$. It follows that 

$$\fG_{(\alpha+1)}\subset \rad(\fG_{(\alpha)})
\subset\fm\subsetneq \fG_{(\alpha)},$$

\noindent which contradicts that $\fG_{(\alpha+1)}=\fG_{(\alpha)}$.
\end{proof}

\section{Conclusion}\label{conclusion}

As a conclusion, we  reduce the Sierra Walton Conjecture to three distinct classes of Lie algebras. 
Then we discuss the status of the conjecture with 
intuitive arguments.

\subsection{Reduction of Sierra-Walton Conjecture to three cases}

\begin{lemma}\label{prepa3cases} Any  Noetherian Lie algebra
$\fG$ contains a characteristic ideal
$\fm$ such that $\fG/\fm$ is just-infinite.

Moreover, if $\fG$ is just infinite, then either

\begin{enumerate} 
\item[(a)]  $[\fG,\fG]$ is residually 
nilpotent, or
\item[(b)] $\cD^\omega\fG$ is a finite 
codimension ideal which is perfect and just-infinite.
\end{enumerate}
\end{lemma}

\begin{proof} By Corollary \ref{idealACC}(a), the Lie algebra
$\fG$ satisfies the ACC condition on ideals. Therefore there exists  a maximal element 
$\fm$ in the set of  ideals 
of infinite codimension. 

By Lemma \ref{closure}(a), $\Cl_\fG(\fm)$ has infinite codimension, therefore the ideal $\fm$ is closed. By Lemma
\ref{closure}(b), the ideal $\fm$ is characteristic,
which proves the first claim.

Assume now that $\fG$ is just infinite.
If $\cD^\omega\fG$ has infinite codimension, then $\cD^\omega\fG=0$ and 
$[\fG,\fG]$ is residually 
nilpotent by Lemma \ref{basic}(f).

Otherwise $\cD^\omega\fG$ is a perfect ideal of finite codimension which is  just-infinite by Lemma \ref{ideal-just}.
\end{proof}

\begin{cor} \label{3cases} Any infinite dimensional Noetherian Lie algebra $\fG$ contains two characteristic ideals $\fp\supset\fq$ with
\begin{align*}
\codim&\, \fp<\infty & \codim\,\fq=\infty
\end{align*} 
\noindent such that the  Lie algebra $\fp/\fq$ 
is just-infinite and  either

\begin{enumerate}
\item[(type A)] $\fp/\fq$ is simple, or

\item[(type B)] $\fp/\fq$ is residually nilpotent, or

\item[(type C)] $\fp/\fq$ is perfect, residually semi-simple
and does not satisfies any polynomial identity.
\end{enumerate} 
\end{cor}

\begin{proof} By Lemma \ref{prepa3cases},
$\fG$ contains a characteristic ideal $\fm$ such that  $\fG/\fm$ is just-infinite. Even if it entails using $\fG/\fm$ instead of $\fG$, we can assume that that $\fm=0$. We will show that
$\fG$ contains a characteristic ideal $\fp$ 
of type A, B, or C.

If $\fp:=[\fG,\fG]$ is residually nilpotent the claim is proved. Thus by Lemma \ref{prepa3cases}(b)
we can assume that 
$\cD^\omega\fG$ is a finite 
codimension ideal which is perfect and just-infinite. Thus, using $\cD^\omega\fG$ instead of
$\fG$ if necessary, we can assume that $\fG$
is perfect.

In $\fG$ any ideal is zero or infinite dimensional,  therefore
$\cD^*\rad^f(\fG)=\cD^{\omega}\rad^f(\fG)$ and there are inclusions

$$\fG_{(0)}=\cD^{\omega}\rad^f(\fG)
\subset\rad^f(\fG)\subset\fG.$$

\begin{enumerate}
\item[(a)] If $\fG_{(0)}\neq 0$, all its simple quotient are infinite dimensional by Theorem A(a). In such a case, we deduce that $\fp:=\fG_{(0)}$ is simple, that is of type A. 

\item[(b)] If $\fG_{(0)}=0$ but $\rad^f(\fG)\neq 0$,
then $\fp:=\rad^f(\fG)$ is residually nilpotent by Lemma \ref{basic}(f), that is of type B.

\item[(c)] If $\fG_{(0)}=\rad^f(\fG)=0$, then
$\fp:=\fG$ is  perfect and residually semisimple. Moreover it does not satisfy any polynomial identity by
Theorem A(c), that is 
$\fp$ is of type $C$.
\end{enumerate}
\end{proof}

In consequence, we obtain

\begin{infcor} 
The Sierra Walton conjecture 
is true  if it is proved for all Lie algebras of types A, B and C.
\end{infcor}

\subsection{Could Sierra-Walton Conjecture be undecidable for some simple Lie algebras?}\label{undecidable}

Let $\alpha$ be a countable  ordinal. We ask

{\it Does there exists a  weakly Noetherian Lie Algebras 
$\fG$ which satisfy
\begin{enumerate}
\item[(a)]  $\G$ is perfect, $\dim\G=\aleph_0$ and 
$\fG/\rad(\fG)$ is a
Novikov-Krichever Lie algebra $\Vect_X$,
\item[(b)] the Lie algebra $\fS:=\fG_{(\alpha)}$ is simple, and
\item[(c)] the natural morphism $\fG\to\Der\,\fS$ is injective?
\end{enumerate}}

{\it The following comments}, which assume a positive answer to the previous question, {\it  are unproved statements:}
Condition (a) implies that $\fG=\fG_{(0)}$. 
Since Theorem A shows that $\fG_{(0)}$ has a very constrained structure, heuristic arguments suggests that we should assume 
that $X$ is a  curve a positive genus. It is likely that  the existence of a large  Lie algebra of outer derivations $\Out(\fS)$ matters in the outcome of the Sierra-Walton conjecture.
Indeed some intuitive arguments  suggest that for $\alpha$ nonrecursive  the left Noetherianity of $U(\fS)$ could be undecidable.

\part*{\hskip35mm Part B: Proof of Theorem B}

\section{$\Z^n$-graded Noetherian Lie algebras}
\label{Z-Noetherian}

\subsection{Reduction to $\Z$-graded Lie algebras}

First, we show that a weakly Noethrian
$\Z^n$-graded Lie algebra $\cL$ admits 
a $\Z$-grading, with homogenous components of finite dimension.

To simplify notations, we denote the $n$-uples
of integers as 
$${\bf m}=(m_1,\ldots,m_n)\in \Z^n.$$
Let $\oplus_{{\bf m}\in\Z^n}\cL_{\bf m}$
be a $\Z^n$-graded Lie algebra. 
Except stated otherwise, it is always assumed that the components $\cL_{\bf m}$ are finite dimensional.

For an additive map $\pi:\Z^n\to\Z$ we define the $\Z$-graded Lie algebra
$\pi_*\cL$ by

$$(\pi_*\cL)_n=\oplus_{\pi({\bf m})=n}
\,\cL_{\bf m}.$$

\noindent In general, the homogenous components
$(\pi_*\cL)_n$ are infinite dimensional.

\begin{lemma}\label{reduc} Let $\cL$ be a weakly 
Noetherian $\Z^n$-graded Lie algebra. Then
there is an additive map 
$\pi:\Z^n\to\Z$ such that

$$\dim (\pi_*\cL)_n<\infty, \,\forall n\in\Z.$$
\end{lemma}

\begin{proof}
First we prove that there is an 
additive map $\pi:\Z^n\to\Z$ such that

$$(\pi_*\cL)_0=\cL_{\bf 0}.$$

The proof runs by induction over $n$.
We write the $n$-uples
of integers as $({\bf m},k)$, where
${\bf m}$ is a $n-1$-uple and $k$ is an integer.
Let 
$$\cL=\oplus_{k\in\Z}\,\cL(k)$$
be the decomposition of the Lie algebra $\cL$
defined by 
$\cL(k)=\oplus_{{\bf m}\in \Z^{n-1}} 
\cL_{({\bf m},k)}.$

By definition, $\cL(0)$
is a $\Z^{n-1}$-graded Lie algebra. Set

$$\Deg\,\cL(0):=
\{{\bf m}\in  \Z^{n-1}\mid 
\cL_{({\bf m},0)}\neq 0\}.$$

By induction
hypothesis, there is an additive map
$\mu:  \Z^{n-1}\to\Z$ such that
$\Ker\,\mu\cap \Deg\,\cL(0)=\{{\bf 0}\}$, where
${\bf 0}$ is the $n-1$-uple $(0,\ldots,0)$.

Next set 
$\cL^\pm:=\oplus_{\pm k>0}\, \cL(k)$. 
The vector spaces 

$$\cL^\pm,\,V^\pm:=\cL^\pm/[\cL^\pm,\cL^\pm]$$

\noindent inherit a $\Z^n$ grading. Set

\begin{align*}
\Deg\,\cL=&\{(\bf m,k)\in \Z^n
\mid \cL_{(\bf m,k)}\neq 0\}\\
\Deg\,\cL^\pm=&\{(\bf m,k)\in \Z^n
\mid \cL^\pm_{(\bf m,k)}\neq 0\}\\
\Deg\,V^\pm=&\{(\bf m,k)\in \Z^n
\mid V^\pm_{(\bf m,k)}\neq 0\}.
\end{align*}

For a positive integer $a$, we define the addive map $\pi:\Z^n\to\Z$ by 
$\pi(\bf m,k)=\mu({\bf m})+ak$. 
Since $V^\pm$ are abelian sections of $\cL$,  the sets
$\Deg\,V^\pm$ are finite. Therefore we can
 choose $a$ big enough such that

$$\pm\pi(\bf m,k)>0, \,\,\forall
(\bf m,k)\in \Deg\,V^\pm.$$

\noindent Since the elements of $\Deg\,\cL^\pm$
are sums of elements in $\Deg\,V^\pm$, we have

$$\pm\pi(\bf m,k)>0,\,\,\forall
(\bf m,k)\in \Deg\,\cL^\pm.$$

\noindent It follows that
$\Ker\pi\cap \Deg\,\cL=\{({\bf 0},0)\}$,
which proves that $(\pi_*\cL)_0=\cL_{\bf 0}$.

We now claim that each component 
$(\pi_*\cL)_n$ is finite dimensional.
It has been proved for $n=0$, so we can assume
$n\neq 0$. We observe that 
$(\pi_*\cL)_n$ is isomorphic to the abelian section 
$\cP/\cQ$, where

\begin{align*}
\cP&=\oplus_{k{\geq 1}}\, (\pi_*\cL)_{kn}\\
\cQ&=\oplus_{k{\geq 2}}\, (\pi_*\cL)_{kn}.
\end{align*}

\noindent Therefore $(\pi_*\cL)_n$ is finite dimensional for all integers $n$.
\end{proof}

\subsection{Weak Noetherianity implies strong Noetherianity} As a consequence of the previous lemma, we deduce:
\begin{cor}\label{strongNoeth}
Any weakly Noetherian $\Z^n$-graded Lie algebra
$\cL$ is strongly Noetherian.
\end{cor}

\begin{proof} By Lemma \ref{reduc}, we can assume that $\cL$ is $\Z$-graded. Set 
$$\cL^\pm=\oplus_{\pm k>0}\cL_{\pm k}.$$

First, we apply Lemma \ref{Noeth-just} to the Lie algebra $\fg=\cL_0\oplus\cL^+$ endowed with the filtration $\fg(0)\subset\fg(1)\subset\ldots$, where $\fg(n)=\oplus_{k=0}^n\,\cL_k$, for any $n\geq 0$. Let $\fG$ be the associated graded Lie algebra. We observe that $\fG_{\geq 1}= \cL^+$ is weakly Noetherian and $\fg(0)=\cL_0$
is strongly Noetherian. Therefore by Lemma 
\ref{Noeth-just}, the Lie algebra $\fg=\cL_0\oplus\cL^+$ is strongly Noetherian.

Exchanging the role of $\cL^+$ and $\cL^-$, we conclude that $\cL_0\oplus\cL^-$ is strongly Noetherian as well.

 Next, we apply again Lemma \ref{Noeth-just} to the Lie algebra $\fg=\cL$ endowed with the filtration $\fg(0)\subset\fg(1)\subset\ldots$, where $\fg(n)=\oplus_{k\leq n}\,\,\cL_k$, for any $n\geq 0$. Let $\fG$ be the associated graded Lie algebra. We observe that $\fG_{\geq 1}= \cL^+$ is weakly Noetherian and we have just proved that 
 $\fg(0)=\cL_0\oplus \cL^-$
is strongly Noetherian. Therefore by Lemma 
\ref{Noeth-just}, the Lie algebra $\fg=\cL$ is strongly Noetherian, which completes the proof.
\end{proof}

We have just proved that, for a $\Z^n$-graded Lie algebra $\cL$,
weak Noetherianity and strong Noetherianity are equivalent. From now  on, we will simply say that 
such a Lie algebra $\cL$ is {\it Noetherian}.

\section{Simple $\Z^n$-graded Lie algebras.}\label{classification}

In this section, we classify
strictly Noetherian simple $\Z^n$-graded Lie algebras.
We deduce that, for any Noetherian 
$\Z^n$-graded Lie algebra $\cL$, we
have $\cL_{(1)}=0$.

\subsection{$L$-forms of $\Witt$ and $\Vir$}

Let $\fG_{\overline{K}}$ be a simple Lie algebra
over $\overline{K}$ and let $L$ be a finite extension of $K$.  A simple Lie algebra
$\fG_L$ with centroid $L$ is called an {\it $L$-form of 
$\fG_{\overline{K}}$} if

$$\G_L\otimes_{L}\,{\overline{K}}\simeq \fG_{\overline{K}}.$$ 

\noindent In general, $L$-forms of $\fG_{\overline{K}}$ are classified by the nonabelian Galois cohomology
$$H^1(\Gal(\overline{K}/L), \Aut(\fG_{\overline{K}})),$$
\noindent see \cite{Serre}.

Assume now that the simple Lie algebra $\fG_{\overline{K}}$ is a 
$\Z$-graded Lie algebra. A simple 
$\Z$-graded Lie algebra
$\G_L$ with centroid $L$ is called a {\it $\Z$-graded $L$-form of $\fG_{\overline{K}}$} if

$$\G_L\otimes_{L}\,{\overline{K}}\simeq \fG_{\overline{K}}
\,\,\text{ as a $\Z$-graded Lie algebra}.$$

Obviously,  $\Z$-graded $L$-forms are classified by 
$$H^1(\Gal(\overline{K}/L), \Aut_0(\fG_{\overline{K}}),$$ 

\noindent where
$\Aut_0(\fG_{\overline{K}})$ is the group of
grading-preserving automorphisms of $\fG_{\overline{K}}$.

\begin{lemma}\label{form} Let $L$ be a finite 
extension of $K$.
\begin{enumerate}
\item[(a)] Any $\Z$-graded $L$-form of $\Witt(\overline{K})$ is isomorphic to $\Witt(L)$, and
\item[(b)] Any $\Z$-graded $L$-form of $\Vir(\overline{K})$ is isomorphic to $\Vir(L)$.
\end{enumerate}
\end{lemma}

\begin{proof} 

Clearly any automorphism
in  $\Aut_0 (\Witt(\overline{K}))$
or in $\Aut_0 (\Vir(\overline{K}))$ is of the form

$$t^n\frac{t\d}{\d t}
\mapsto a^n t^n\frac{t\d}{\d t},$$

\noindent for some nonzero scalar $a$.
Therefore 

$$\Aut_0 (\Witt(\overline{K}))=\Aut_0 (\Vir(\overline{K}))=\overline{K}^*.$$

Since, by Hilbert 90 Theorem \cite{Hilbert90}, 

$$H^1(\Gal(\overline{K}/L), 
\overline{K}^*)=\{1\},$$

\noindent we deduce that 
$\Witt(\overline{K})$ and
$\Vir(\overline{K})$ admits only one $\Z$-graded $L$-form,
namely $\Witt(L)$ and $\Vir(L)$, which completes the proof.
\end{proof}

\subsection{Cartan subalgebras and splitting fields.}
Let $\cL=\oplus \cL_n$ be a $\Z$-graded Lie algebra. As usual, we assume that all
components $\cL_n$ are finite dimensional.
A subalgebra $\fh\subset \cL_0$ is called 
a {\it Cartan subalgebra} if $\fh$ is nilpotent 
and $N_{\cL_0}(\fh)=\fh$. 

For any $\alpha\in \fh^*$, set 

\begin{align*}
\mathcal{L}^{(\alpha)}&:=
\{x\in \mathcal{L}\mid \, (\ad(h)-\alpha(h))^N(x) = 0 \quad
\forall h\in\fh \text{ and } N\gg0\}\\
 \end{align*}
 
The spaces $\mathcal{L}^{(\alpha)}$ are called the {\it generalized $\fh$-weightspaces}. They admit a decomposition 

$$\mathcal{L}^{(\alpha)}=\oplus_{n\in\Z}\mathcal{L}^{((\alpha,n))},$$

\noindent where $\mathcal{L}^{((\alpha,n))}=\cL_n\cap \mathcal{L}^{(\alpha)}$. We say that the Cartan subalgebra $\fh$ is {\it split} if  $\cL$ is a direct sum of generalized weightspaces,
that is:
 
 $$\cL=\oplus_{\alpha\in\fh^*}\,\,\mathcal{L}^{(\alpha)}.$$

\begin{lemma}\label{split} Let $\cL$ be a finitely generated
$\Z$-graded Lie algebra, and let $\fh$ be a Cartan subalgebra of $\cL_0$.
\begin{enumerate} 
\item[(a)] There exists a finite extension $E$ of $K$ such that $\fh\otimes E$ is a split Cartan subalgebra of $\cL\otimes E$.
\item[(b)] Assume moreover that $\cL$ is simple with centroid $L$.  
Then 
\begin{enumerate}
\item[(b1)] $L$ is a finite extension of $K$,
\item[(b2)] $\fh$ is a $L$-vector subspace of $\cL_0$, and
\item[(b3)] there exists a finite extension $E$ of $L$ such that $\fh\otimes_L E$ is a split Cartan subalgebra of $\cL\otimes_L E$.
\end{enumerate}
\end{enumerate}
\end{lemma}

\begin{proof} We prove Assertion (a). By hypothesis, there is 
a finite dimensional $\fh$-module $M$ which generates $\cL$. There is a finite extension
$E$ of $K$ such that 
the $\fh\otimes E$-module $M\otimes E$ is a direct sum of generalized weightspaces. Set $\cL_E=\cL\otimes E$. Since
$$[\cL^{(\alpha)}, \cL^{(\beta)}]\subset \cL^{(\alpha+\beta)}
\,\,\forall \alpha,\beta\in\fh^*,$$

\noindent we concude that $\cL_E$ is a direct sum of generalized weight spaces, therefore $\fh\otimes E$ is a split Cartan subalgebra.

We now prove Assertion (b1). We claim that each
homogenous component $\cL_n$ is a $L$-vector space.
Otherwise, we can find an element $a\in L$
of degree $k\neq 0$, that is satisfying
 $$a\cL_n\subset \cL_{n+k},\,\,\forall n\in \Z.$$
It follows easily  that 
$\cL$ is a free $K[a,a^{-1}]$-module of finite rank,
which contradicts that
$1+a$ is  invertible. Therefore each  $\cL_n$ is a $L$-vector space, which implies that $L$ is a finite extension of $K$.

For the proof of Assertion (b2), write $\tilde{\fh}$ for the $L$-vector space generated by $\fh$. Clearly $\tilde{\fh}$ is a nilpotent Lie algebra. Since $N_{\tilde h}(\fh)=\fh$, we deduce that $\tilde{\fh}=\fh$, which means that $\fh$ is a $L$-vector space. 
Moreover Assertion (b3) follows from Assertion (a) applied to 
$\cL$, viewed as a
Lie algebra over $L$. 
\end{proof}

\subsection{Rank of simple $\Z$-graded Lie algebras}
We now assume that $\cL$ is a finitely generated $\Z$-graded simple Lie algebra with centroid $L$ and let $\fh\subset\cL_0$ be a Cartan subalgebra. 

For any extension $E$ of $L$, set
\begin{align*}
\cL_E&:=\cL\otimes_L E\\
\fh_E&:=\fh \otimes_L E
\end{align*}

The field $E$ is called a {\it splitting field} if
$\fh_E$ is a split Cartan subalgebra.
By Lemma \ref{split}(b2), there are finite extensions  $E$ of $L$
which are splitting fields. Given a splitting field,
the {\it set of roots}  of 
$\mathcal{L}_E$ is:

$$\Delta\coloneqq \{\widetilde{\alpha}\in \fh^*_E\times\Z\mid \cL_E^{\widetilde\alpha}\neq 0\},$$

\noindent where, for  $\widetilde\alpha=(\alpha,n)\fh^*_E\times\Z$,

$$\cL_E^{(\widetilde\alpha)}:=\cL_E^{(\alpha)}\cap (\cL_E)_n.$$

\noindent We will see that $\cL_0\neq 0$ and, with our nonstandard definition, $(0,0)$ is a root, called {\it the trivial root}.
A root $\widetilde{\alpha}=(\alpha,n)$ is called
{\it real} if $\alpha\neq 0$ and  {\it imaginary} otherwise. Let $\Delta_{\operatorname{re}}$ be the set of real roots.

The {\it root lattice} is the subgroup
$Q\subset \fh_E^*\times \Z$ generated by $\Delta$. 
As an abstract group, $Q$ is independent of the choice of a splitting field. We define the rank 
$\rk\cL$ of
$\cL$ as the rank of $Q$. When $\cL$ is infinite dimensional,
we have obviously $\rk\cL\geq 1$. Indeed  it is also true when $\cL$ is finite dimensional, see  Lemma \ref{imageCSA}.
Since  it is assumed that $\cL$ is finitely generated,
$\rk\cL$ is finite.

\smallskip
\begin{remark} Assume that
$\fh$ is a split Cartan subalgebra of
$\cL$. Any element $x\in\cL_0^{(\alpha)}$
with $\alpha\neq 0$ acts locally nilpotently on $\cL$, thus 

$$\exp\ad(x)$$ 

\noindent is a well defined automorphism in $\Aut_0(\cL)$.
The subgroup $\Elem(\cL)\subset\Aut_0(\cL)$ generated by those automorphisms is called the group of {\it elementary automorphism of $\cL$}. It follows from \cite{Bourbaki} that any two split Cartan subalgebras of ${\mathcal L}_0$ are conjugated by an elementary automorphism. 

Therefore $Q$ and  $\rk\cL$
are indeed independent of the choice of a split Cartan subalgebra of ${\mathcal L}_0$. Since we do not need
this fact, we will not provide more details.
\end{remark}

\subsection{Simple Lie algebras of rank one}
\label{rank1}

Let $\cL$ be a finitely generated infinite dimensional 
$\Z$-graded simple Lie algebra 
of rank one,  let $\fh$ be a Cartan subalgebra of $\cL_0$. Set $L=\Cent(\cL)$ and let $E$ be a splitting field.

By definition of the rank, there is 
$\tilde{\alpha}=(\alpha, 1)\in 
(\fh_E^*,\Z)$ such that

$$\cL_E=\oplus 
\cL_E^{(n\tilde{\alpha})}.$$

There are two cases
\begin{enumerate}
\item[(a)]$\alpha=0$, that is all roots are imaginary, and
\item[(b)]$\alpha\neq0$, that is all roots are reals, except the trivial root.
\end{enumerate}

The first case is impossible:

\begin{lemma}\cite[Lemma 22]{Mathieu-jalg} 
\label{V}
Let $\cL$ be a simple $\Z$-graded Lie algebra. If  $\cL_{\overline{K}}$ is finitely generated,  then $\Delta$ contains
at least on real root.
\end{lemma}

We turn now our attention to the second case, that is $\alpha\neq 0$.
The following result is implicitly proved in
\cite{Mathieu-inv2}, and more detailled account is given in \cite[theorem 5.8]{AndMath}. The appearance of a free Lie algebra in  Assertion (b) is connected with the
Gabber-Kac Theorem \cite{GabberKac} for the contragredient Lie algebra $G(^{2\,2}_{2\,2})$.

\begin{lemma}\label{S}
Let $\cL$ be a 
finitely generated simple $\Z$-graded Lie algebra of rank one. 
Then 
\begin{enumerate}
\item[(a)] Either $\cL_{\overline{K}}$ is isomorphic to $\fsl_2(\overline{K})$, $\Witt(\overline{K})$, or $\Vir(\overline{K})$, 
\item[(b)] or $\cL_{\overline{K}}$ contains 
a nonabelian free Lie algebra.
\end{enumerate} 
\end{lemma}

We deduce:

\begin{cor}\label{rk1} Let $\cL$ be a simple $\Z$-graded Lie algebra of rank one. If $\cL$ is 
strictly Noetherian
and infinite dimensional, then

$$\cL\simeq\Witt(L)\rm{\,\,or\,\,} 
\cL\simeq\Vir(L),$$

\noindent for some finite extension $L$ of $K$.
\end{cor}

\begin{proof} By By Corollary \ref{strongNoeth}, $\cL$ is finitely generated. 

If  $\cL_{\overline K}$ contains a noncommutative free Lie algebra, it contains a free
 Lie algebra over two generators $x$ and $y$.
 These generators belong to
$\cL_E$ for some finite extension
$E$ of $K$. By Lemma \ref{basic}(c), 
$\cL_E$ is not Noetherian, which contradicts that $\cL$ is strictly Noetherian.

Therefore, by Lemma \ref{S}, 
$\cL_{\overline{K}}$ is isomorphic to 
$\Witt(\overline{K})$, or $\Vir(\overline{K})$.
Thus by Lemma \ref{form}, we have

$$\cL\simeq\Witt(L)\rm{\,\,or\,\,} 
\cL\simeq\Vir(L),$$

\noindent where $L$ is the centroid of $\cL$.
\end{proof}

\subsection{Simple Lie $\overline{K}$-algebras of rank $\geq 2$ }

\medbreak
\noindent Let $\fG$ be a simple 
$\Z$-graded Lie algebra of rank $\geq 2$
over $\overline{K}$ and let $\fh$ be a Cartan subalgebra of $\fG_0$.
We now define two hypothetical properties, and we will see that
any such $\fG$ satisfies one of them.
By the end of the section it will be clear that these properties are mutually exclusive.

\medbreak
To start with, we define the notion of a string. Let $\widetilde\alpha\in Q$ and 
$\widetilde\beta\in \Delta$.
There are $a,b\in\Z\cup\{\pm\infty\}$
with $a<0<b$ such that

\begin{enumerate}[leftmargin=*,label=\rm{(\roman*)}]
\item $\widetilde\beta+k\widetilde\alpha$ belongs to $\Delta$ for any $k\in]a,b[$, but
	
\medbreak
\item neither 
$\widetilde\beta+a\widetilde\alpha$ nor
$\widetilde\beta+b\widetilde\alpha$ belongs to 
$\Delta$.
\end{enumerate}

\medbreak
\noindent The set 
$\{\widetilde\beta+k\widetilde\alpha\mid k\in]a,b[\}$ is called the 
\emph{$\widetilde\alpha$-string through $\widetilde\beta$}. Obviously, the string is infinite if
$a=-\infty$ or $b=\infty$.

\medbreak
The first hypothetical property  
$(\mathcal{H}_{\operatorname{re}})$ is the following: 
\begin{align}\tag{$\mathcal{H}_{\operatorname{re}}$} 
\begin{aligned}
&\text{There exist }\widetilde{\alpha}\in\Delta_{\operatorname{re}}, \quad
\widetilde{\beta}\in \Delta, \quad 
\widetilde{\beta} \notin \Q.\widetilde{\alpha},  \text{ such that}
\\ &\text{the $\widetilde\alpha$-string through $\widetilde\beta$ is infinite.}
\end{aligned}
\end{align}

\medbreak
The second hypothetical property is the notion of weak integrability.
Following \cite{Mathieu-inv2}, we say that $\mathcal{\G}$ is {\it weakly integrable} if,
for any $\widetilde{\alpha}\in\Delta_{\operatorname{re}}$, we have 
\begin{align*}
\bigcap_{n\geq 0}\,\Ad(\mathcal{\G}^{\widetilde\alpha})^n(\mathcal{\G})=0.
\end{align*}

The following result has been proved in
\cite{AndMath}, see Lemma 6.5.

\begin{lemma}\label{dicho} Let $\mathcal{\fG}$ be a simple $\Z$-graded $\overline{K}$-algebra of
rank $\geq 2$. Then either

\begin{enumerate}[leftmargin=*,label=\rm{(\alph*)}]
\item   $\mathcal{\fG}$ satisfies the hypothesis 
$(\mathcal{H}_{\operatorname{re}})$, or

\item  $\mathcal{\fG}$ is weakly integrable.
\end{enumerate}
\end{lemma}

The following result has been  proved in \cite{Mathieu-inv1}
and  \cite{Mathieu-inv2}, see Theorem 4.

\begin{theorem}\label{integrability}  If $\mathcal{\fG}$ is weakly integrable, then $\fG$ is an affine Lie algebra.
\end{theorem}

\bigskip{\it Remark} The Lie algebra 
$\fG$ is called {\it integrable} if
for any real root $\tilde{\alpha}$, 
$ad(\fG^{(\tilde\alpha)})$ is locally nilpotent.
This definition of integrability is more natural, and Theorem \ref{integrability} 
is proved in \cite{Mathieu-inv1} under the condition that $\fG$ is integrable.

However the notion of weak integrability is more
adapted to the proofs, and it is proved
in \cite{Mathieu-inv2} that weak integrability implies integrability.

\subsection{Simple Lie algebras of rank $\geq 2$ over the nonalgebraically closed field $K$}

The previous subsection did involve
simple $\Z$-graded Lie algebras of rank $\geq 2$
over $\overline{K}$. 

We now investigate
$\Z$-graded Lie algebras $\cL$ of rank $\geq 2$
over $K$. We also assume that $\cL$ is simple as a non-graded Lie algebra.

\begin{lemma}\label{Nork2} No simple finitely generated simple Lie algebra 
$\cL$ of rank $\geq 2$ is
strictly Noetherian.
\end{lemma}

\begin{proof} Let $\cL$ be a simple finitely generated simple Lie algebra of rank $\geq 2$ and let $\fh$ be a Cartan subalgebra of $\cL_0$. 

Any affine Lie algebra is not simple as an abstract
Lie algebra, and its centroid is infinite dimensional,
therefore $\cL_{\overline K}$ is not affine.
 By Lemma \ref{dicho}, we conclude that $\cL_{\overline K}$  satisfies hypothesis 
$(\mathcal{H}_{\operatorname{re}})$.

By Lemma \ref{split},  there exists a finite extension $E$ of $L$ such that 
$\fh_E$ is split. We observe that
$\cL_E$ admits the same root system as 
$\cL_{\overline K}$. Therefore there is a real root
$\tilde{\alpha}\in \Delta_{\operatorname{re}}$
and another  root 
$\tilde{\beta}\not\in\Q\alpha$ such that the $\tilde{\alpha}$-string going through $\tilde{\beta}$ is infinite.
Set 

\begin{align*}
\fp=&\oplus_{k=1}^\infty
\oplus_{n\in\Z}\, \cL_E^{k\tilde{\beta}+n\tilde{\alpha}}\\
\fq=&\oplus_{k=2}^\infty
\oplus_{n\in\Z}\, \cL_E^{k\tilde{\beta}+n\tilde{\alpha}}
\end{align*}

\noindent Obviously,
$\fp/\fq\simeq  
\oplus_{n\in\Z}\, \cL_E^{\tilde{\beta}+n\tilde{\alpha}}$ is an infinite dimensional abelian section of $\cL_E$. Thus
$\cL$ is not strictly Noetherian.
\end{proof}

\subsection{Strictly Noetherian simple $\Z^n$-graded Lie algebras}

As a conclusion of the  section, we obtain

\begin{theorem}\label{simplequot} Let $\cL$ be an
infinite-dimensional $\Z^n$-graded
Lie algebra over $K$. 

If $\cL$ is simple and strictly Noetherian, 
then $\cL$ is isomorphic to
$\Witt(L)$ or $\Vir(L)$ for some 
finite extension $L$ of $K$.
\end{theorem}

\begin{proof} By Lemma \ref{reduc}, we can assume that $n=1$.
By Lemma \ref{Nork2}, the Lie algebra has rank one. Thus 
the Theorem follows from Corollary \ref{rk1}.
\end{proof}

\begin{cor}\label{vanishing} Let $\cL$ be a strictly Noetherian
$\Z^n$-graded Lie algebra. Then
$$\cL_{(1)}=0.$$
\end{cor}

\begin{proof} Assume otherwise. By
Corollary \ref{idealACC}, $\cL_{(1)}$ satisfies the ACC on ideals, therefore $\cL_{(1)}$ admits a maximal ideal. By Theorem \ref{characteristic}, $\fm$ is a graded ideal. 

 Thus by Theorem \ref{simplequot}, $\cL_{(1)}/\fm$ is either isomorphic to
$\Witt(L)$ or $\Vir(L)$  or $\cL_{(1)}/\fm$ is finite dimensional.

This contradicts  Assertions (d) and (e) of Theorem A,
which asserts that no simple quotient of
$\cL_{(1)}$ is a finite dimensional Lie algebra or  is a Krichever-Novikov
Lie algebra. We conclude that $\cL_{(1)}=0$.
\end{proof}

\section{The structure Theorem for perfect Noetherian graded Lie algebras}\label{proofB}

Let $\cL$ be a strictly Noetherian
$\Z^n$-graded Lie algebra. 

In Corollary 
\ref{vanishing}, we have already shown that
$\cL_{(1)}=0$. 
In Subsection \ref{L0/L1}, we determine
the structure of 
$\cL_{(0)}\simeq \cL_{(0)}/\cL_{(1)}$.
Furthermore, assume that $\cL$ is perfect.
In Subsection \ref{secL/L0}, we show that
$\cL/\cL_{(0)}$ is finite dimensional.
We conclude in the last subsection that,
up to isogeny, $\cL$ is a direct sum of
$\cL/\cL_{(0)}$ and $\cL_{(0)}$, which
proves Theorem B.

\subsection{Simple quotients of $\cL$}
\label{L0/L1}

In this subsection, we show that any strictly Noetherian
$\Z^n$-graded Lie algebra $\cL$ admits only finitely many maximal ideals of codimension $>1$. Then we deduce the structure of
$\cL_{(0)}\simeq \cL_{(0)}/\cL_{(1)}$.

\begin{lemma}\label{imageCSA} Let 
$\fs=\oplus_{n\in \Z}\, \fs_n$ be a finite dimensional simple Lie algebra endowed with a $\Z$-grading.

Then any Cartan subalgebra $\fh$ of $\fs_0$ is a
Cartan subalgebra of $\fs$.
\end{lemma}

\begin{proof} Since by Cartan theorem \ref{cartan}(e),
we have $\Out(\fs)=0$, there is an element $D\in \fs$ such that

$$[D,x]=nx\,\,\forall x\in \fs_n.$$

\noindent Clearly, $D$ belongs to $\fs_0$. Since
$D$ is central in $\fs_0$, it belongs to
$\fh$. 

Since $N_D(\fs)=\fs_0$, we conclude that

$$N_\fs(\fh)=N_{\fs_0}(\fh)$$

\noindent therefore $\fh$ is a Cartan subalgebra of $\fs$.
\end{proof}

\begin{lemma}\label{finite} Let $\cL$ be a strictly Noetherian
$\Z^n$-graded Lie algebra.
Then $\cL$ admits only finitely many maximal ideals.
\end{lemma}

\begin{proof} By Lemma \ref{reduc},
we can assume that $\cL$ is $\Z$-graded.
Let $\fh$ be a Cartan subalgebra of $\cL_0$,
let $\fm_1,\ldots,\fm_m$ be a finite family of distinct maximal ideals, and let
$\fh_i$ be the image of $\fh$ in $\cL/\fm_i$.

By Theorem \ref{characteristic}, we observe that all ideals
$\fm_i$ are graded ideals. When $\cL/\fm_i$ is finite dimensional,
$\fh_i$ is a Cartan subalgra of $\cL/\fm_i$ by
Lemma \ref{imageCSA}.
Otherwise, $\cL/\fm_i$ is isomorphic to
$\Witt(L)$ or $\Vir(L)$ by Proposition
\ref{simplequot} and $\fh_i$ is isomorphic to $L$, the zero part of $\Witt(L)$ or $\Vir(L)$.

Therefore the number of maximal ideals is bounded by $\dim \fh$.
\end{proof}

Let $E$ be a finite extension of $K$. Recall that 
$\Witt(E)$ is centrally closed \cite{Fuks} and that
$\widehat{\Vir(E)}$ is a central extension of
$\Vir(E)$ by $E$ \cite{Fuks}.

\begin{cor}\label{L0} We have

 $$
\widehat{\cL}_{(0)}\simeq 
\,\big[\oplus_{i=1}^n\Witt(E_i)\big]\,\oplus 
\,\big[\oplus_{j=1}^m\widehat{\Vir}(F_j)\big]
$$

\noindent where 
$n$, $m$ are integers and 
$E_1,\ldots,E_n,F_1,\ldots, F_m$ are finite extensions of $K$.
\end{cor}

\begin{proof} By Lemma \ref{finite} and Theorem \ref{simplequot},
we have
\begin{align*}
\cL_{(0)}/\rad(\cL_{(0)})\simeq 
\,\big[\oplus_{i=1}^n\Witt(E_i)\big]\,\oplus 
\,\big[\oplus_{j=1}^m{\Vir}(F_j)\big]
\end{align*}

\noindent for some integers 
$n$, $m$ and some finite extensions 
$E_1,\ldots,E_n,F_1,\ldots, F_m$ of $K$.
By Theorem A (b), $\cL_{(0)}/\cL_{(1)}$ is a central extension
of $\cL_{(0)}/\rad(\cL_{(0)})$ and by 
Corollary \ref{vanishing} we have
$\cL_{(1)}=0$. It follows that

$$
\widehat{\cL}_{(0)}\simeq 
\,\big[\oplus_{i=1}^n\Witt(E_i)\big]\,\oplus 
\,\big[\oplus_{j=1}^m\widehat{\Vir}(F_j)\big].
$$
\end{proof}

\subsection{Quasi-minuscule weights}

Let $\fs$ be a finite dimensional semi-simple Lie algebra,
and let $\fh$ be a split Cartan subalgebra.
We implicity assume that a Borel subalgebra is given, therefore the notion of
positive roots and dominant weights are well
defined.  For any dominant weight, we denote by
$L(\omega)$ the simple module with highest weight 
$\omega$.

First, assume that $\fs$ is simple and let
$W$ be its Weyl group. If the set of weight
of $L(\omega)$ is $W.\omega)$
or $W.\omega\cup\{0\}$, the weight $\omega$ is called {\it quasi-minuscule}. The latter case occurs only when $\omega$ is the highest  root if $\fs$ is of type ADE, or the highest short
root otherwise.

We will now extend the classical definition 
of quasi-minuscule weights to the
 case where $\fs$ is semi-simple, that is

$$\fs=\oplus_{i=1}^m\,\fs_i,$$

\noindent where each summand $\fs_i$ is simple.
A  weight $\mu$ is called 
{\it quasi-minuscule} if

$$\mu=\sum_{i\in I}\,\omega_i,$$

\noindent where $I$ is a non-empty subset
of $\{1,\ldots,m\}$ and each $\omega_i$ is a
quasi-minuscule weight for $\fs_i$.

\begin{lemma}\label{quasiminuscule}
Let $V$ be a finite dimensional $\fs$-module. Assume that 

\centerline{$V^{(\mu)}=0$}
 
 \noindent for any  quasi-minuscule weight $\mu$. Then $V$ is a trivial $\fs$-module, that is $\fs.V = 0$.
\end{lemma}

\begin{proof}  Clearly the statement amounts to the fact that whenever $V$ is a nontrivial
simple $\fs$-module $V$, we have $V^{(\mu)}\neq 0$, for some  quasi-minuscule weight $\mu$. 

Indeed $V$ is of the form

$$V=\otimes_{i\in I}\, V_i,$$

\noindent where $I$ is a non-empty subset
of $\{1,\ldots,m\}$ and each $V_i$ is a 
nontrivial simple $\fs_i$-module. Therefore for
$j\not\in I$, $\fs_j$ acts trivially.

By \cite{Bourbaki} {Chapter 8}, for each
$i\in I$, there is a quasi-minuscle weight 
$\omega_i$ such that $V_i^{(\omega_i)}\neq 0$.
Thus we have

$$V^{(\omega)}\neq 0,$$

\noindent where $\omega$ is the quasi-minuscule weight $\omega=\sum_{i\in I}\,\omega_i$.
\end{proof}

\subsection{Proof that $\cL/\cL{(0)}$ is finite dimensional}
\label{secL/L0}

We will now use that $\cL$ is assumed to be perfect.

\begin{lemma}\label{stable} Let $\cL$ be a  strictly 
Noetherian perfect $\Z$-graded Lie algebra.
Then 

\begin{itemize}
\item[(a)] $\cL/\rad^f(\cL)$ is finite dimensional, and
\item[(b)] there is an integer $M$ such that
$$\cC^{M+1} \rad^f(\cL)=
\cC^M \rad^f(\cL).$$
\end{itemize}
\end{lemma}

\begin{proof} By Lemma \ref{finite}, $\cL$ admits only finitely many maximal ideals, thus $\cL/\rad^f(\cL)$ is finite dimensional, which proves Assertion (a)

We now prove Assertion (b). 
Let $\fh$ be a Cartan subalgebra 
of  $\cL_0$. By Lemma \ref{split}(a), 
there is a finite extension $E$ of $K$ such that
$\fh\otimes E$ is a split Cartan subalgebra
in $\cL\otimes E$. 

We observe that  

$$\rad^f(\cL\otimes E)=\rad^f(\cL)\otimes E.$$

\noindent Even if it means choosing $\cL\otimes E$ instead of $\cL$, we may assume that $\fh$ is split.

Therefore $\cL$ admits a generalized weight decomposition

$$\cL=\oplus_{\alpha\in \fh^*}\cL^{(\alpha)},$$

\noindent where $\cL^{(\alpha)}=\{x\in\cL\mid 
(\ad(h)-\alpha(h))^N(x)=0 {\rm\,\, for\,} N\gg0\}$.

We claim that $\cL^{(\alpha)}$ is finite dimensional whenever $\alpha\neq 0$.
Set $\fp=\oplus_{k\geq 1}\,\cL^{(k\alpha)}$
and $\fq=\oplus_{k\geq 2}\,\cL^{(k\alpha)}$.
Then $\cL^{(\alpha)}$ is isomorphic to the abelian section $\fp/\fq$, which proves the claim.

Set $\fr=\rad^f(\cL)$, $\fs=\cL/\fr$
and, for any $k\geq 1$, set 
$V_k=\cC^k\fr/\cC^{k+1} \fr$.
We have already proved that $\fs$ is a 
(finite dimensional) semisimple Lie algebra.
Since they are abelian sections, each component $V_k$ is a finite dimensional $\fs$-module.

Since the Lie algebra $\cL/[\fr,\fr]$ is finite dimensional,
Theorem \ref{cartan}(b) implies that

$$\cL/[\fr,\fr]\simeq \fs\ltimes V_1.$$

\noindent Since $\cL$ is perfect
and $V_1$  is an abelian ideal, we deduce that

\begin{align}
V_1&=[\fs, V_1].
\end{align}

Let $\overline{\fh}$ be the image of $\fh$
in $\fs$. Since each  $V_k$ is an $\fs$-module,
$\fh$ acts diagonally  on $V_k$ 
and the action factors trough $\overline{\fh}$.
Thus for any quasi-minuscule weight 
$\omega$ we can consider the weightspace

$$\,V_k^{(\omega)}=\{x\in V_k\mid h.x=\omega(h)x
\,\forall h\in\fh\}.$$

We have proved that $\cL^{(\omega)}$ is finite dimensional, henceforth 

$$\oplus_{k{\geq 1}}
\,V_k^{(\omega)}$$

\noindent is finite dimensional. Therefore there
is an integer $N$, such that

$$V_k^{(\omega)}=0,$$

\noindent for all $k\geq N$ and all
quasi-minuscule weight $\omega$. By Lemma
\ref{quasiminuscule}, we obtain that
$\fs$ acts trivially on $V_k$ for $k\geq N$.

The Lie bracket
$\fr\times \cC^N\,\fr\to C^{N+1}\,\fr$
induces a surjective
$\fs$-equivariant map

$$\mu: V_1\otimes V_N
 \to V_{N+1}.$$
 
 \noindent We have proved that $V_1=\fs.V_1$
 and $\fs.V_N=\fs.V_{N+1}=0$. It follows that
 
 \begin{align*}
 V_{N+1}&=\mu(V_1\otimes V_N)\\
 &=\mu(\fs.V_1\otimes V_N)\\
 &=\fs.\mu(V_1\otimes V_N)\\
 &=\fs.V_{N+1},
 \end{align*}
 
\noindent from which  we conclude that $V_{N+1}=0$, that is
 
 $$\cC^{M}\,\fr=\cC^{M+1}\,\fr,$$
\noindent where $M=N+1$, which completes the proof.
\end{proof}

\begin{cor}\label{L/L0}
Let $\cL$ be a  strictly 
Noetherian perfect $\Z$-graded Lie algebra.

Then $\cL/\cL_{(0)}$ is finite dimensional.
\end{cor}

\begin{proof} By Lemma \ref{stable}(a),
$\rad^f(\cL)$ has finite codimension,
therefore by Lemma \ref{basic}(g) we have

$$\cD^\omega\rad(\cL)=\cC^\omega\rad(\cL)$$

By Lemma \ref{stable}(b), there is an integer $M$ 
such that
$$\cC^\omega\rad(\cL)=\cC^M\rad(\cL).$$

Hence the descending derived series stabilizes. It implies that $\cC^M\rad(\cL)$ is perfect, which means that

$$\cL_{(0)}=\cC^M\rad(\cL),$$

\noindent which implies the claim.
\end{proof}

\subsection{Proof of Theorem B}

We finally prove  the following result:

\begin{mainB}
Let $\cL$ be a perfect strictly Noetherian
$\Z^n$-graded Lie algebra
and let $\widehat{\cL}$ be its its universal central extension. Then we have

\begin{align*}
\widehat{\cL}\simeq \fg\,\oplus 
\,\big[\oplus_{i=1}^n\Witt(E_i)\big]\,\oplus 
\,\big[\oplus_{j=1}^m\widehat{\Vir}(F_j)\big]
\end{align*}

\noindent where $\fg\simeq \fs\ltimes\fr$ is a  (perfect) finite dimensional Lie algebra,
$n$, $m$ are integers and 
$E_1,\ldots,E_n,F_1,\ldots, F_m$ are finite extensions of $K$.
\end{mainB}

\begin{proof} Let
$\fz$ be the center of $\cL_{(0)}$.
By Corollary \ref{L0},  $\cL_{(0)}/\fz$ is a direct sum 
of Witt and Virasoro Lie algebras. Thus by
Grabowski Theorem \ref{Grabowski}, we have
$\Out(\cL_{(0)}/\fz)=0$. Since $\cL$ and $\cL_{(0)}$ are perfect,
by Lemma \ref{centralizer}(b),
we have

$$\widehat{\cL}=\widehat{\cL/\cL_{(0)})}+\widehat{\cL_{(0)}}.$$

\noindent By Corollary \ref{L/L0}, 
$\fg:=\widehat{\cL/\cL_{(0)})}$ is finite dimensional.
  Moreover by Corollary \ref{L0}, we have

 \begin{align*}
\widehat{\cL}_{(0)}\simeq 
\,\big[\oplus_{i=1}^n\Witt(E_i)\big]\,\oplus 
\,\big[\oplus_{j=1}^m\widehat{\Vir}(F_j)\big]
\end{align*}

\noindent for some integers
$n$, $m$  and some finite extensions
$E_1,\ldots,E_n,F_1,\ldots, F_m$  of $K$. This completes the proof of the Theorem.
\end{proof}

\begin{cor}\label{SW}
Let $\cL$ be a perfect 
$\Z^n$-graded Lie algebra. 
The algebra $U(\cL)$ is left Noetherian
only if $\cL$ has finite dimension.
\end{cor}

\begin{proof} Assume that $U(\cL)$ is left Noetherian.
Then by Corollary \ref{logical}  the Lie algebra 
$\cL$ is strictly  Noetherian. 
By Sierra Walton Theorem \ref{SWthm}, $\cL$ does not admit a section
isomorphic to $\Witt(K)$. Thus by Theorem B,
$\cL$ is finite dimensional.
\end{proof}

\subsection*{Acknowledgements} 
The author  thanks Efim Zelmanov and Slava Futorny 
for the warm hospitality during their visit  to the 
Shenzhen International Center for Mathematics.
He also thanks Ivan Shestakov for interesting discussions
about polynomial identities. 

Initially, the author was skeptical about the existence of general structure results for  infinite dimensional Lie algebras,
beside the case of locally finite dimensional Lie algebras
treated by I. Stewart and its coauthors.  Special thanks are due to Nicolas Andruskiewitsch, who 
had raise the issue repeatedly.

\end{document}